\documentclass[reqno,12pt]{amsart}
\usepackage[left=1.2in, right = 1.2in, top=1in]{geometry}
\usepackage{amsmath, amssymb, amsthm, mathrsfs, color, times, textcomp, verbatim}
\usepackage{hyperref}
\usepackage{amstext}
\usepackage{appendix}
\usepackage{verbatim}
\usepackage{cases}
\usepackage{pgfplots}
\usepackage{tikz}
\usepackage{fancyhdr}
\usepackage{graphicx} 

\makeatletter

\newtheorem{theorem}{Theorem}[section]

\newtheorem{lem}{Lemma}[section]

\newtheorem{defn}{Definition}[section]

\makeatother

\begin{document}
\title{A priori interior estimates for special Lagrangian curvature equations }
\author{Guohuan Qiu }
\author{Xingchen Zhou}
\date{}
\address{Institute of Mathematics, Academy of Mathematics and Systems Science, Chinese Academy of Sciences, No.55 Zhongguancun East Road, 100190, Beijing, China}
\email{qiugh@amss.ac.cn}

\address{Department of Mathematical Sciences, Tsinghua University, Beijing, China}
\email{zxc3zxc4zxc5@stu.xjtu.edu.cn}

\maketitle

\begin{abstract}
We establish a priori interior curvature estimates for the special Lagrangian curvature equations in both the critical phase and convex case. Additionally, we prove a priori interior gradient estimates for any constant phases.
\end{abstract}


\section{Introduction}\label{sec:intro}

The special Lagrangian equation 
\begin{equation}
 \sum_{i=1}^{n}\arctan\lambda_{i} (D^2 u)=\Theta  \label{SLE}
\end{equation}
was introduced by Harvey-Lawson \cite{harvey1982calibrated} back in 1982. Its solutions $u$ were shown to have the property that the graph $(x,\nabla u) \in \mathbb{R}^n \times \mathbb{R}^n $ is a Lagrangian submanifold which is absolutely volume-minimizing. In this paper, we are going to develop some analytical properties of the special Lagrangian curvature equation \eqref{eq:specialLag} which has similar structure as \eqref{SLE}. 

Given a hypersurface $M \in \mathbb{R}^n$, we denote its position vector by $X$ and  $\nu$ an outer normal vector. At any point $X \in M$, the  principal curvature $\kappa=(\kappa_{1},\kappa_{2},\cdots,\kappa_{n})$ satisfy an equation 
\begin{equation}
\sum_{i=1}^{n}\arctan\kappa_{i}=\Theta.
\label{eq:specialLag}
\end{equation}
Equation \eqref{eq:specialLag} was first studied by Graham Smith in \cite{smith2013special}, where he referred to it as the "special Lagrangian curvature" equations. He introduced the notion of special Lagrangian curvature as an alternative higher-dimensional generalization of two or three-dimensional scalar curvature. He also demonstrated how this curvature is related to the canonical special Legendrian structure of spherical subbundles of the tangent bundle of the ambient manifold. He then established a compactness result if $\Theta \in [\frac{(n-1)\pi}{2}, \frac{n\pi}{2})$.

Harvey-Lawson, in \cite{harvey2021pseudoconvexity}, investigated the existence of so-called "special Lagrangian potential equations"  given by
\begin{equation}
    \sum^m_{i=1} \arctan(\lambda^g_i(A)) =\Theta \quad for \quad A\in Sym^2(\mathbb{R}^n), \label{SLPE}
\end{equation}
 where $g : Sym^2(\mathbb{R}^n) \rightarrow \mathbb{R}$ is a homogeneous polynomial of degree $m$. As mentioned in their paper, besides the special Lagrangian equation \eqref{eq:specialLag}, one important example of \eqref{SLPE} is the deformed Hermitian-Yang-Mills equation, which appears in mirror symmetry. Another interesting example of \eqref{SLPE} is the special Lagrangian curvature equation \eqref{eq:specialLag} from hypersurface geometry.
 
 Moreover, our motivation for studying the special Lagrangian curvature equation comes from an observation made in Section \ref{OTinD2} that in dimension two, this equation \eqref{eq:specialLag} is equivalent to the equation in the optimal transportation problem with a "relative heat cost" function, as discussed in Brenier's paper \cite{BrenierLecNote01}.
 
Several years ago, Warren-Yuan \cite{WY10} and Wang-Yuan \cite{WY11} successfully derived a priori interior Hessian and gradient estimates for the special Lagrangian equations \eqref{SLE} for both critical and supercritical phases. Chen-Warren-Yuan also obtained results for the convex case in \cite{chen2009priori}. In this work, we aim to extend their results by proving the following theorem, which generalizes their findings to the special Lagrangian curvature equations \eqref{eq:specialLag}.
 
\begin{theorem}
\label{thm:SpecialLag} Consider a smooth hypersurface $M$ in $\mathbb{R}^{n+1}$, which can be parametrized as a graph over $B_{10} \subset \mathbb{R}^{n}$ for $n \geq 3$. The graph function is assumed to be a solution to the equation (\ref{eq:specialLag}). If $\Theta = \frac{(n-2)\pi}{2}$ or $M$ is convex, then we have
\begin{equation}
\sup_{x\in B_{\frac{1}{2}}}|\kappa(x)|\leq C, \label{InteriorCurevature}
\end{equation}
where $C$ depends only on $||M||_{L^\infty(B_{10})}$ and $n$. 
\end{theorem}
Motivated by Weyl isometric embedding problem, E. Heinz in \cite{heinz1959elliptic} first established these interior curvature estimates \eqref{InteriorCurevature} to the scalar curvature equation in $\mathbb{R}^2$. The special Lagrangian curvature equations are fully nonlinear elliptic partial differential equations. Heinz's interior $C^{2}$ estimate is a highly non-trivial result for fully nonlinear PDEs. This is because Pogorelov \cite{pogorelov1978multidimensional} constructed non-$C^2$ convex solutions to the equation $\det D^{2}u=1$ in $B_{r}\subseteq\mathbb{R}^{n}$
when $n\geq3$. Pogorelov's counter-examples were extended by Urbas in \cite{urbas1990existence} to $\sigma_k$ hessian equation when $n\geq k \geq 3$.
When $k=2$, a pioneer work was done by Warren-Yuan \cite{warren2009hessian} where they obtained $C^{2}$ interior estimate for the equation
\begin{equation*}
\sigma_{2}(\nabla^{2}u)=1,\begin{array}{ccc}
 & x\in B_{1}\subset\mathbb{R}^{3}.\end{array}
\end{equation*}
In order to show an interior $C^2$ estimate for $\sigma_2 (\nabla^{2}u)=f$ with a smooth $f$  in dimension three, the first-named author established a strong trace Jacobi inequality and a doubling lemma in \cite{qiu2017interior}. The trace Jacobi inequality and the doubling lemma are crucial components in the recent resolution of the four-dimensional case by Shankar-Yuan in \cite{shankar2023hessian}. While the paper \cite{mcgonagle2019hessian,guan2019interior,shankar2020hessian} have established various results under specific convexity conditions, the higher-dimensional case ($n\geq 5$) of $2$-Hessian equations remains an open question.

In hypersurface geometry, numerous curvature equations are formulated as fully nonlinear PDEs. Sheng-Urbas-Wang in \cite{ShengUrbasWang04} proved interior curvature estimates for a class of curvature equations given by $\frac{\sigma_k(\kappa)}{\sigma_{k-1}(\kappa)} = f$. The first-named author in \cite{qiu2019interior} established interior curvature estimates for the scalar curvature equation, corresponding to the case $\Theta=\frac{\pi}{2}$ in dimension three. For higher dimensions scalar curvature equations, Guan and the first-named author in \cite{guan2019interior} obtained estimate \eqref{InteriorCurevature}  with certain convexity constraints. Moreover, we succeeded in establishing a higher-dimensional version of Heinz's compactness result for isometrically immersed hypersurfaces in $\mathbb{R}^n$ with positive scalar curvature in the same paper \cite{guan2019interior}.

An interesting fact is that these special Lagrangian curvature equations \eqref{eq:specialLag} also enjoy a priori interior curvature estimates,  similar to the special Lagrangian equations \eqref{SLE}.
A key observation is that the graph $\Sigma=(X,\nu)$
where $X$ satisfies this equation (\ref{eq:specialLag}) can be
viewed as a submanifold in $(\mathbb{R}^{n+1}\times\mathbb{S}^{n},\sum\limits _{i=1}^{n+1}dx_{i}^{2}+i_{\mathbb{S}^{n}}^{*}(\sum\limits _{i=1}^{n+1}dy_{i}^{2}))$
with bounded mean curvature.  These submanifold with bounded mean curvature allows us to use  Michael-Simon's mean value inequality in \cite{michael1973sobolev}. Another important ingredient  in this problem is establishing a strong Jacobi inequality. We establish a strong trace Jacobi inequality for the special Lagrangian curvature equations in critical and convex cases. Finally, we apply a modified integral method of Warren-Yuan \cite{warren2009hessian,WY10} and Wang-Yuan  \cite{WY11} without relying on Sobolev inequality. 

The general supercritical case is unclear to us. Although all the other parts work fine in the supercritical case, we can only prove the Jacobi inequality on the critical or convex case in our Theorem \ref{lem3}. This is because the curvature term may disrupt the Jacobi inequality, which is crucial for the interior estimate. In contrast to the special Lagrangian equation, we encounter a term like $\frac{|A|^2}{H} G^{ij}h_{ij}$ for this curvature equation. However, $G^{ij}h_{ij}$ may be negative in the supercritical case when $\kappa_1 \geq \cdots \geq \kappa_{n-1} \rightarrow + \infty$ and $\kappa_n \approx -1$. Therefore, this troublesome term is even more challenging than the terms from the cutoff function. We note that a similar bad term also appears in Guan-Ren-Wang's paper \cite{guan2015global}, where they derived a global $C^2$ a priori estimate for convex solutions of curvature equations with a prescribed function that depends on $\nu$.
Instead, we conjecture that there may be a $C^{1,\alpha}$ singular solution to the special Lagrangian curvature equation for supercritical case, particular when $0<\Theta<\frac{\pi}{2}$, in $\mathbb{R}^2$. Because the equation can also be derived from an optimal transportation problem, as shown Lemma \ref{optimalTransPro} in section \ref{OTinD2}, but it violates the famous Ma-Trudinger-Wang condition when $0<\Theta<\frac{\pi}{2}$. According to Ma,Trudinger,Wang's paper \cite{ma2005regularity} and  Loeper's paper \cite{Loeper09Acta}, this condition is necessary and sufficient to guarantee that the equation has interior second-order estimates for general optimal transportation problems. Note that we can not use Loeper's general counterexample for this particular curvature equation. 

In section \ref{gradientestimateSec}, we prove the interior gradient estimate for all constant phases. 
\begin{theorem}
Suppose $M$ is a smooth graph over $B_{1}\subset\mathbb{R}^{n}$
and it is a solution of equation (\ref{eq:specialLag}). Then we have
\begin{equation*}
|Du(0)|\leq C(n)\mathop{osc}\limits_{B_1(0)}u.
\end{equation*}
\end{theorem}
We remark that for the special Lagrangian equation \eqref{SLE}, especially for sub-critical phases, the interior gradient estimate remains open, as mentioned in \cite{mooney2022homogeneous}.
Due to the existence of singular solutions for the special Lagrangian equation \eqref{SLE} in subcritical phases ($|\Theta|<\frac{(n-2)\pi}{2}$), as demonstrated by Nadirashvili-Vladu \cite{nadirashvili2010singular}, Wang-Yuan \cite{wang2013singular}, and Mooney-Savin \cite{mooney2023non} for non $C^1$ examples, we do not expect the special Lagrangian curvature equation to have interior curvature estimates when $|\Theta|<\frac{(n-2)\pi}{2}$. 
\\

\section{Preliminary Lemmas}

We first introduce some definitions and notations. 
\begin{defn}
\label{defsigma}For $\kappa=(\kappa_{1},\cdots,\kappa_{n})\in\mathbb{R}^{n}$,
the $k$-th elementary symmetric function $\sigma_{k}(\kappa)$
is defined as 
\[
\sigma_{k}(\kappa):=\sum\kappa_{i_{1}}\cdots\kappa_{i_{k}},
\]
 where the sum is taken over for all strictly increasing sequences $i_{1},\cdots,i_{k}$
of the indices chosen from the set $\{1,\cdots,n\}$. And we denote
$\sigma_{0}=1$. The definition can be extended to symmetric matrices
where $\kappa=(\kappa_{1},\cdots,\kappa_{n})$ are the corresponding
eigenvalues of the symmetric matrices. 
\end{defn}
For example, in $\mathbb{R}^{3}$ 
\[
\sigma_{2}(\kappa):=\kappa_{1}\kappa_{2}+\kappa_{1}\kappa_{3}+\kappa_{2}\kappa_{3}.
\]

\begin{defn}
For $1\leq k\leq n$, let $\Gamma_{k}$ be a cone in $\mathbb{R}^{n}$
determined by 
\[
\Gamma_{k}=\{\kappa\in\mathbb{R}^{n}:\sigma_{1}(\kappa)>0,\cdots,\sigma_{k}(\kappa)>0\}.
\]
\end{defn}

\begin{defn}
A $C^{2}$ surface $M$ is called admissible if at every point $X\in M$,
its principal curvature satisfies $(\kappa_{1},\kappa_{2},\cdots,\kappa_{n})\in\Gamma_{n-1}$.
\end{defn}

We remark that $M$ is admissible when $\Theta \ge \frac{(n-2)\pi}{2}$ or $M$ convex, according to Wang-Yuan \cite{WY11}, Lemma 2.1. Thus for a curvature equation (\ref{eq:specialLag}), we may assume that
$M$ is admissible without loss of generality.

For any symmetric matrix $h_{ij}$, it follows that $\sigma_{k}^{ij}:=\frac{\partial\sigma_{k}(\kappa(h_{ij}))}{\partial h_{ij}}$
is positive definite if $\kappa(h_{ij})\in\Gamma_{k}$, for each $1\le k\le n$. 

We also have an important algebraic  property.
\begin{lem}
 Suppose the ordered real numbers $\kappa_{1}\geq\kappa_{2}\geq\cdots\geq\kappa_{n}$ with $\kappa \in \Gamma_{n-1}$,
then 
\begin{equation}
\kappa_{i}+(n-i)\kappa_{n}\geq0,\label{eq:kappan-1}
\end{equation}
for each $1\le i\le n-1$.
\end{lem}
\begin{proof}
We only need to prove it when $\kappa_n<0$. Notice that $\kappa_{n-1}>0$, we have
$$
\frac{\sigma_{n-1}}{\kappa_1\cdots\hat \kappa_i\cdots \kappa_{n-1}}\ge 0\Rightarrow \kappa_i+ \sum^{n-1}_{l=i}\frac{\kappa_i}{\kappa_l}\kappa_n\ge 0.
$$
Thus we have 
\begin{equation*}
\kappa_i + (n-i)\kappa_n \geq \kappa_i+ \sum^{n-1}_{l=i}\frac{\kappa_i}{\kappa_l}\kappa_n\ge 0.
\end{equation*}

\end{proof}

We recall some elementary facts about hypersurface $X=(x,u(x))\in \mathbb{R}^{n+1}$. Denoting 
\begin{equation*}
    W=\sqrt{1+\lvert Du\rvert^{2}},
\end{equation*}
the second fundamental form and the first fundamental form of the hypersurface can be written in local coordinate as 
\begin{equation*}
    h_{ij}=\frac{u_{ij}}{W}
\end{equation*}
and 
\begin{equation*}
    g_{ij}=\delta_{ij}+u_{i}u_{j}.
\end{equation*}
The inverse of the first fundamental form and the Weingarten Curvature are 
\begin{equation*}
    g^{ij}=\delta_{ij}-\frac{u_{i}u_{j}}{W^{2}}
\end{equation*}
and
\begin{equation*}
    h_{i}^{j}=D_{i}(\frac{u_{j}}{W}).
\end{equation*}

\begin{defn}
The Newton transformation tensor is defined as 
\begin{equation*}
    [T_{k}]_{i}^{j}:=\frac{1}{k!}\delta_{jj_{1}\cdots j_{k}}^{ii_{1}\cdots i_{k}}h_{j_{1}}^{i_{1}}\cdots h_{j_{k}}^{i_{k}},
\end{equation*}
and the corresponding $(2,0)$-tensor is defined as 
\begin{equation*}
    [T_{k}]^{ij}:=[T_{k}]_{p}^{i}g^{pj}.
\end{equation*}

In particularly, 
\begin{equation*}
   [T_{0}]^{ij}=g^{ij}. 
\end{equation*}
If $k \notin \{0,1,2,\cdots n \}$, we define
\begin{equation*}
    [T_k]^{ij}=0.
\end{equation*}
\end{defn}
From this definition one can easily show a divergence free identity
\[
\sum_{j}\partial_{j}[T_{k}]_{i}^{j}=0.
\]

\begin{lem} \label{lem7}
There is a family of elementary relations between $\sigma_{k}$ operators
and Newton transformation tensors

\begin{equation}
[T_{k}]_{i}^{j}  =  \sigma_{k}\delta_{i}^{j}-[T_{k-1}]_{i}^{l}h_{l}^{j},\label{eq:Tk1}
\end{equation}
or 
\begin{equation}
[T_{k}]_{i}^{j}=\sigma_{k}\delta_{i}^{j}-[T_{k-1}]_{l}^{j}h_{i}^{l}.\label{eq:Tk2}
\end{equation}
Moreover, the $(2,0)$-tensor of $T_{k}$ is symmetry such that 
\begin{equation}
[T_{k}]^{ij}=[T_{k}]^{ji}.\label{eq:Tij}
\end{equation}
\end{lem}

\begin{proof}
From Definition \ref{defsigma}, it is easy to check that 
\begin{equation}
\sigma_{k}(\kappa)=\frac{1}{k!}\delta_{j_{1}\cdots j_{k}}^{i_{1}\cdots i_{k}}h_{j_{1}}^{i_{1}}\cdots h_{j_{k}}^{i_{k}}.\label{eq:matrix}
\end{equation}
By definition and (\ref{eq:matrix}), we obtain (\ref{eq:Tk1}) as
follows: 
\begin{eqnarray*}
[T_{k}]_{i}^{j} & = & \frac{1}{k!} \delta_{jj_{1}\cdots j_{k}}^{ii_{1}\cdots i_{k}}h_{j_{1}}^{i_{1}}\cdots h_{j_{k}}^{i_{k}}\\
 & = & \frac{1}{k!} (\delta^{i}_{j} \delta_{j_{1}\cdots j_{k}}^{i_{1}\cdots i_{k}} - \delta^{i_1}_{j} \delta_{j_{1} j_2 \cdots j_{k}}^{ii_{2}\cdots i_{k}} - \delta^{i_2}_{j} \delta_{j_{1} j_2 \cdots j_{k}}^{i_1 i\cdots i_{k}} + \cdots )h_{j_{1}}^{i_{1}} h^{i_2}_{j_2} \cdots h_{j_{k}}^{i_{k}}\\
 & = & \frac{1}{k!}  \delta_{j_{1}\cdots j_{k}}^{i_{1}\cdots i_{k}} h_{j_{1}}^{i_{1}}\cdots h_{j_{k}}^{i_{k}} \delta^{i}_{j} - \frac{k}{k!} \sum_{\substack{i_2\cdots i_k \\ j_1 \cdots j_k }} \delta_{j_{1} j_2 \cdots j_{k}}^{ii_{2}\cdots i_{k}} h_{j_{1}}^{j} h_{j_2}^{i_2} \cdots h_{j_{k}}^{i_{k}} \\
 & = & \sigma_{k}\delta_{i}^{j}-[T_{k-1}]_{i}^{k}h_{k}^{j}.
\end{eqnarray*}
And we can also obtain \eqref{eq:Tk2} in the same way.\\
For $k=1$, the symmetry of the $(2,0)$-tensor of $T_{1}$ is obviously
from the symmetry of $h$. Inductively, we assume the symmetry of
$T_{k}$ is true when $k=m$. From (\ref{eq:Tk1}), we have 
\begin{eqnarray*}
[T_{m+1}]^{ij} & = & [T_{m+1}]_{l}^{i}g^{lj}=\sigma_{m+1}\delta_{l}^{i}g^{lj}-[T_{m}]_{l}^{p}h_{p}^{i}g^{lj}\\
 & = & \sigma_{m+1}g^{ij}-[T_{m}]^{pj}h_{p}^{i}.
\end{eqnarray*}
On the other hand, by (\ref{eq:Tk2}) we have 
\begin{eqnarray*}
[T_{m+1}]^{ji} & = & [T_{m+1}]_{l}^{j}g^{li}=\sigma_{m+1}\delta_{l}^{j}g^{li}-[T_{m}]_{p}^{j}h_{l}^{p}g^{li}\\
 & = & \sigma_{m+1}g^{ji}-[T_{m}]^{jp\prime} g_{p p^\prime} h^{p}_l g^{li}\\
 & = & \sigma_{m+1}g^{ji}-[T_{m}]^{jp}h_{p}^{i}.
\end{eqnarray*}
So from the symmetry of $g$ and $T_{m}$, we have proved (\ref{eq:Tij}).
\end{proof}

The algebraic form of the special Lagrangian curvature equation \eqref{eq:specialLag} is sometimes useful and is given by: 
\begin{equation}
F(\kappa):=\cos\Theta \sum\limits _{1\leq 2k+1 \leq n}(-1)^{k}\sigma_{2k+1}(\kappa(x))-\sin \Theta \sum\limits _{0 \leq 2k \leq n}(-1)^{k}\sigma_{2k}(\kappa(x))=0. \label{eq:specialLag2}
\end{equation}

It can sometimes be more useful to express the equation in the following form:
\begin{lem} \label{lem_theta}
Denote $\theta := \Theta - \frac{(n-1)\pi}{2}$, then we have
\begin{equation}
     F(\kappa)= \cos \theta   \sum\limits_{0\leq 2k \leq n} (-1)^{k} \sigma_{n-2k} - \sin \theta \sum\limits_{1 \leq 2k+1 \leq n} (-1)^{k} \sigma_{n-2k-1}=0.\label{eq:SpecialLag3}
\end{equation}
   
\end{lem}
\begin{proof}
If $n=2m+1$, it is easy to see that
\begin{eqnarray*}
    \cos \Theta &=& (-1)^{m} \cos \theta  \\
    \sin \Theta &=& (-1)^{m} \sin \theta.
\end{eqnarray*}
Then the equation \eqref{eq:specialLag2} can be written into
\begin{eqnarray*}
     F(\kappa) & = & (-1)^m \cos \theta   \sum\limits_{0\leq 2k \leq n-1} (-1)^{\frac{n-2k-1}{2}} \sigma_{n-2k} -(-1)^m  \sin \theta \sum\limits_{1 \leq 2k+1 \leq n} (-1)^{\frac{n-2k-1}{2}} \sigma_{n-2k-1}\\
     & =& \cos \theta   \sum\limits_{0\leq 2k \leq n-1} (-1)^{k} \sigma_{n-2k} - \sin \theta \sum\limits_{1 \leq 2k+1 \leq n} (-1)^{k} \sigma_{n-2k-1}=0.
\end{eqnarray*}  
If $n=2m$, it is easy to see that
\begin{eqnarray*}
    \cos \Theta &=& (-1)^{m} \sin \theta,  \\
    \sin \Theta &=& (-1)^{m-1} \cos \theta.
\end{eqnarray*}
Then the equation \eqref{eq:specialLag2} can be written into
\begin{eqnarray*}
    F(\kappa) & = &(-1)^m \sin \theta  \sum\limits_{1\leq 2k+1 \leq n-1} (-1)^{\frac{n-2k-2}{2}} \sigma_{n-2k-1} - (-1)^{m-1} \cos \theta \sum\limits_{0\leq 2k \leq n} (-1)^{\frac{n-2k}{2}} \sigma_{n-2k}\\
    &=& - \sin \theta   \sum\limits_{1\leq 2k+1 \leq n-1} (-1)^k \sigma_{n-2k-1} +\cos \theta \sum\limits_{0\leq 2k \leq n} (-1)^k \sigma_{n-2k}=0.
\end{eqnarray*}
\end{proof}

We show that $V:=\cos \Theta \sum\limits_{0 \leq 2k \leq n} (-1)^k \sigma_{2k} + \sin \Theta \sum\limits_{1\leq 2k+1 \leq n} (-1)^k  \sigma_{2k+1} $ is positive and satisfies
\begin{lem}  
    \begin{equation*}
        V\geq 1.
    \end{equation*}
\end{lem}

\begin{proof}
   We denote 
    \begin{equation*}
        V_1:= \sum\limits_{0 \leq 2k \leq n} (-1)^k \sigma_{2k} 
    \end{equation*}
    and
     \begin{equation*}
        V_2:= \sum\limits_{1\leq 2k+1 \leq n} (-1)^k  \sigma_{2k+1}.
    \end{equation*}
Our equation \eqref{eq:specialLag} and $V$ are 
\begin{equation*}
    0= \cos \Theta V_2 - \sin \Theta V_1
\end{equation*}
\begin{equation*}
    V= \sin \Theta V_2+ \cos \Theta V_1 .
\end{equation*}
Thus we have 
\begin{equation*}
    V_1 = \cos \Theta V
\end{equation*}
and 
\begin{equation*}
    V_2 = \sin \Theta V.
\end{equation*}
We have directly
\begin{eqnarray*}
 \Pi_{i=1}^{n}(1+\sqrt{-1} \kappa_{i}) & = &   \sum_{k=0}^{n}(\sqrt{-1})^k \sigma_{k}\\
& =&  \sum\limits_{0 \leq 2k \leq n} (-1)^k \sigma_{2k} +  \sum\limits_{1 \leq 2k+1 \leq n} \sqrt{-1} (-1)^k \sigma_{2k+1}\\
& = &  V_1 +\sqrt{-1}  V_2. 
\end{eqnarray*}
Thus by the equation \eqref{eq:specialLag}, we have
\begin{eqnarray*}
\cos \Theta \Pi_{i=1}^{n}(1+\sqrt{-1} \kappa_{i})-\sqrt{-1} \sin \Theta \Pi_{i=1}^{n}(1+\sqrt{-1}  \kappa_{i})= V.
\end{eqnarray*}
So we only need to prove the left hand side is positive when $u$
satisfies the equation (\ref{eq:specialLag}). Denoting $\theta_{i}=\arctan \kappa_{i}\in(-\frac{\pi}{2},\frac{\pi}{2})$,
we have 
\begin{eqnarray*}
\Pi_{i=1}^{n}(1+\sqrt{-1} \kappa_{i}) & = & \Pi_{i=1}^{n}(1+\sqrt{-1}\tan\theta_{i})\\
 & = & \Pi_{l=1}^{n}\cos^{-1}\theta_{l}\Pi_{i=1}^{n}(\cos\theta_{i}+\sqrt{-1}\sin\theta_{i})\\
 & = & \Pi_{l=1}^{n}\cos^{-1}\theta_{l}[\cos(\sum_{i=1}^{n}\theta_{i})+\sqrt{-1}\sin(\sum_{i=1}^{n}\theta_{i})]\\
 & = & \Pi_{l=1}^{n}\cos^{-1}\theta_{l}(\cos \Theta +\sqrt{-1}\sin \Theta).
\end{eqnarray*}
Thus
\begin{eqnarray*}
V&= & \cos\Theta [\Pi_{l=1}^{n}\cos^{-1}\theta_{l}(\cos \Theta +\sqrt{-1}\sin \Theta)] - \sqrt{-1} \sin\Theta [\Pi_{l=1}^{n}\cos^{-1}\theta_{l}(\cos \Theta +\sqrt{-1}\sin \Theta)]\\
&=&\Pi_{l=1}^{n}\cos^{-1}\theta_{l} \geq 1.
\end{eqnarray*}
\end{proof}

Moreover, $V$ is actually the volume form of the submanifold $\Sigma = (X,\nu) $.
\begin{lem} \label{Lem9}
There is a following elementary identity 
\begin{equation*}
    \prod_{i=1}^{n}(1+\kappa_{i}^{2})=(\sum\limits_{0 \leq 2k \leq n} (-1)^k \sigma_{2k} )^{2}+(\sum\limits_{1\leq 2k+1 \leq n} (-1)^k  \sigma_{2k+1})^{2}.
\end{equation*}
Moreover, if $X$ satisfies the equation (\ref{eq:specialLag}),
then 
\begin{eqnarray*}
    V=\sqrt{\det G}.
\end{eqnarray*} 
\end{lem}

\begin{proof}
 \begin{eqnarray*}
     \sqrt{\det G} &=& \sqrt{\prod_{i=1}^{n}(1+\kappa_{i}^{2})}\\
     &=& \sqrt{\prod_{i=1}^{n} (1+\sqrt{-1} \kappa_i)(1-\sqrt{-1} \kappa_i)}\\
     &=& \sqrt{V^2_1+V^2_2}\\
     &=& \lvert V \rvert = V.
 \end{eqnarray*}
\end{proof}

Let us denote $F^{ij}:=\frac{\partial F}{\partial h_{ij}}$ , and we are
going to prove that $F^{ij}$ is positive definite when $u$ satisfies
(\ref{eq:specialLag}).
\begin{lem} 
The first fundamental form of the graph $X^{\Sigma}=(X,\nu)\in \mathbb{R}^{n+1} \times \mathbb{R}^{n+1}$  is $G_{ij}=<X_{i}^{\Sigma},X_{j}^{\Sigma}>=g_{ij}+h_{i}^{k}h_{kj}$.
And if $u$ satisfies the equation (\ref{eq:specialLag2}), the inverse of $G_{ij}$
is 
\begin{equation}
    G^{ij}=\frac{F^{ij}}{V}.\label{GF}
\end{equation}
\end{lem}

\begin{proof}

We have $X_{i}^{\Sigma}=(X_{i},h_{i}^{k}X_{k})$, then the first fundamental form is 
\begin{equation*}
G_{ij}=<X_{i}^{\Sigma},X_{j}^{\Sigma}>=g_{ij}+h_{i}^{k}h_{kj}.
\end{equation*}
And by definition, we have 
\begin{equation*}
F^{ij}= \cos \Theta \sum\limits _{0 \leq 2k\leq n-1} (-1)^{k}[T_{2k}]^{ij} -\sin \Theta \sum\limits _{1 \leq 2k-1 \leq n-1} (-1)^k [T_{2k-1}]^{ij}.
\end{equation*}
Due to an elementary identity
\begin{equation*}
[T_{k}]^{ip}h_{p}^{j}=-[T_{k+1}]^{ij}+\sigma_{k+1}g^{ij},
\end{equation*}
we have 
\begin{eqnarray*}
F^{iq}G_{qj} & = & \cos\Theta \sum\limits _{0 \leq 2k \leq n-1}(-1)^{k}[T_{2k}]^{iq}(g_{qj}+ h_{q}^{p}h_{pj}) -\sin\Theta \sum\limits_{1\leq 2k-1 \leq n-1} (-1)^k [T_{2k-1}]^{iq} (g_{qj}+h_{q}^{p}h_{pj})\\
&= & \cos \Theta \sum\limits_{0 \leq 2k \leq n-1} (-1)^k \big \{-[T_{2k-1}]^{ip} h_{pj}+ \sigma_{2k} \delta^i_j +\sigma_{2k+1} h^{i}_j - [T_{2k+1}]^{ip}h_{pj} \big\} \\
& &-\sin \Theta \sum\limits_{1\leq 2k-1 \leq n-1} (-1)^k \big\{ -[T_{2k-2}]^{ip} h_{pj} +\sigma_{2k-1} \delta^i_j +\sigma_{2k} h^i_j -[T_{2k}]^{ip} h_{pj}   \big\}\\
&=&  \big\{ \cos \Theta \sum\limits_{0 \leq 2k \leq n} (-1)^k \sigma_{2k}  +\sin \Theta \sum\limits_{1\leq 2k+1 \leq n} (-1)^k  \sigma_{2k+1} \big\} \delta^i_j .
\end{eqnarray*}
Here we have used the equation \eqref{eq:specialLag2} in the last equality.
So we have verified that 
\[
G^{ij}=\frac{\cos \Theta \sum\limits _{0 \leq 2k\leq n-1} (-1)^{k}[T_{2k}]^{ij} -\sin \Theta \sum\limits _{1\leq 2k-1\leq n-1} (-1)^k [T_{2k-1}]^{ij}}{\cos \Theta \sum\limits_{0 \leq 2k \leq n} (-1)^k \sigma_{2k}  +\sin \Theta \sum\limits_{1\leq 2k+1 \leq n} (-1)^k  \sigma_{2k+1}}=\frac{F^{ij}}{V}.
\]
\end{proof}

\section{An important differential inequality}
For convenience, we choose an orthonormal frame in $\mathbb{R}^{n+1}$ such that $\{e_{1},e_{2},\cdots,e_{n}\}$ are tangent to $M$ and $\nu$ is outer normal on $M$. Let us denote the dual form and the connection form to be $w^i$ and $w^j_i$, where $i$ and $j$ range from $1$ to $n$. The second fundamental form is denoted by $h_{ij}$. We have the moving frame formulas on $M$
\begin{eqnarray*}
  dX &=& w^i e_i \\
  d \nu &=& \sum_j h_{ij}  w^i e_j \\
  de_i &=& w^j_i e_j - h_{ij} w^j \nu.
\end{eqnarray*}
Due to $<e_i,e_j>=\delta_{ij}$, we have 
\begin{equation*}
    w^j_i + w^i_j=0.
\end{equation*}
The structure equations on $M$ are
\begin{eqnarray*}
  dw^i &=& w^j \wedge  w^i_j, \\
  d w^i_j &=& w^k_j \wedge w^i_k - h_{jk} w^k \wedge h_{il} w^l.
\end{eqnarray*}
The curvature tensor is defined by
\begin{equation*}
    \frac{1}{2} R_{ijkl} w^k \wedge w^l := dw^i_j - w^k_j \wedge w^i_k.
\end{equation*}
For a function $f: M\rightarrow \mathbb{R}$, we define its covariant derivative by 
\begin{eqnarray*}
  f_i  &:=& df(e_i), \\
  f_{ij} &:=& df_i(e_j) - w^k_i(e_j) f_k, \\
  f_{ijk} &:=& df_{ij}(e_k)-w^l_i (e_k) f_{lj}-w^l_j (e_k) f_{il},\\
   f_{ijkl} &:=& df_{ijk}(e_l)-w^p_i (e_l) f_{pjk}-w^p_j (e_l) f_{ipk}-w^p_k (e_l) f_{ijp}.
\end{eqnarray*}
Thus we have the following fundamental formulas of a hypersurface in $\mathbb{R}^{n+1}$:
\begin{eqnarray*}
X_i &=& e_i,\\
X_{ij} & = & -h_{ij}\nu ,\begin{array}{cc}
 & (Gauss\,\,formula)\end{array}\\
\nu_{i} & = & h_{ij}\delta^{jk} e_{k},\begin{array}{cc}
 & (Weingarten\,\,equation)\end{array}\\
h_{ijk} & = & h_{ikj},\begin{array}{cc}
 & (Codazzi\,\,equation)\end{array}\\
R_{ijkl} & = & h_{ik}h_{jl}-h_{il}h_{jk}.\begin{array}{cc}
 & (Gauss\,\,equation)\end{array}
\end{eqnarray*}
We also have the following commutator formula: 
\begin{eqnarray}
h_{ijkl}-h_{ijlk} & = & \sum_m h_{im}R_{mjkl}+\sum_m h_{mj}R_{mikl}.\label{eq:commute}
\end{eqnarray}

Combining Codazzi equation, Gauss equation and (\ref{eq:commute}), we have 
\begin{eqnarray}
h_{iikk}=h_{kkii}+\sum_{m}(h_{im}h_{mi}h_{kk}-h_{mk}^{2}h_{ii}).\label{eq:commute2}
\end{eqnarray}

On the submanifold $X^\Sigma = (X,\nu) \in \mathbb{R}^{n+1} \times \mathbb{R}^{n+1}$, we denote $\{ e_i^{\Sigma} := (e_i,\nu_i) \}^n_{i=1}$ which are tangent vectors on $\Sigma$. And $\nu_1 :=(\nu,0)$, $\nu_2 :=(0,\nu)$, $\overline{e_i} := (-\nu_i, e_i)$ are normal vectors on $\Sigma$, where $i$ ranges from $1$ to $n$.  Using this frame $\{e^\Sigma_i,\nu_1, \nu_2, \overline{e_{i}}\}$, we have 
$G_{ij} = <e_i^{\Sigma},e_j^{\Sigma}> = \delta_{ij} + h_{ik}\delta^{kl} h_{lj}.$
The moving frame formulas on $X^\Sigma$ are
\begin{eqnarray*}
    dX^\Sigma &=& w^i e^\Sigma_i, \\
    de^\Sigma_i &=& (w^j_i e_j - h_{ik}w^k \nu, \sum\limits_j h_{ijk}w^k e_j +  w^p_i h_{pj} e_j -h_{ij} h_{jk} w^k \nu )\\
    &=& \sum_j (w^j_i+h_{ipl}G^{kp} h_{kj} w^l ) e^\Sigma_j +G^{jk} h_{ikl}w^l  \overline{e_j} - h_{ik} w^k \nu_1 - \sum_j h_{ij}h_{jk}w^k \nu_2. 
\end{eqnarray*}
Thus, the Levi-Civita connection form of $G_{ij}$ on $\Sigma$ is
\begin{equation*}
    (w^{\Sigma})^j_i := w^j_i+h_{ipl}G^{kp} h_{kj} w^l  .
\end{equation*}
We have the following lemma:
\begin{lem} Any function $f: M\rightarrow \mathbb{R}$ can also be regarded as a function on $\Sigma$. We have the equation
\begin{equation} 
    \triangle_{G} f =G^{ij} f_{ij}, \label{1stEQ}
\end{equation}
where $f_{ij}$ represents the second covariant hessian of $f$ on $M$, and $\triangle_G$ denotes the Laplacian with respect to the metric $G$.
\end{lem}
\begin{proof}
\begin{eqnarray*}
    \triangle_{G} f &=& G^{ij} (\partial_{ij}f-(w^\Sigma)^k_i(e_j) f_k)\\
    &=& G^{ij} (\partial_{ij}f-w^k_i(e_j) f_k +w^k_i(e_j) f_k-  (w^\Sigma)^k_i(e_j) f_k)\\
    &=& G^{ij} f_{ij} - G^{ij} h_{ipl} G^{pq}h_{qk}w^l(e_j) f_k   \\
    &=&G^{ij} f_{ij} - G^{ij} h_{ipj} G^{pq}h_{qk} f_k.  
\end{eqnarray*}
Because
\begin{equation*}
    G^{ij} h_{ipj} =0,
\end{equation*}
we obtain the identity \eqref{1stEQ}.
     
\end{proof}

The following diagonal lemma is very useful. It is employed in Shankar-Yuan \cite{SY22} and Zhou \cite{zhou2023hessian}. For completeness, we also provide a proof.
\begin{lem} \label{Diag}
Let $a_1,\cdots a_{m}$ be positive constants, $ (b_1,\cdots,b_{m})\in \mathbb{R}^{m}$ and let $ x:=(x_1,\cdots,x_m)$ be the unknowns. 
The quadratic polynomial
\begin{equation*}
    Q(x) = \sum\limits^m_{i=1} a^2_i x^2_i - (\sum\limits^m_{i=1}  b_i x_i)^2 \geq 0
\end{equation*}
is equivalent to the following inequality,
\begin{equation*}
    1-\sum^{m}_{i=1}\frac{b^2_i}{a^2_i}\ge 0.
\end{equation*}
\end{lem}

\begin{proof}
Denote $L:= (b_1,\cdots,b_m)\in \mathbb{R}^m$ and $E_i = (0,\cdots,\overset{i}{1},\cdots,0)$. 
In order to study the positivity of the quadratic  $Q$, we consider the following $m\times m$ symmetric matrix,
\begin{equation*} 
\Lambda=\sum^{m}_{i=1}a^2_i E_i^T E_i-L^TL.
\end{equation*}
Denote 
\begin{equation*}
    \overline{L}:= \sum^m_i \frac{b_i}{a_i} E_i
\end{equation*}
and
\begin{equation*}
A := \left[    
\begin{array}{lllll}
 \frac{1}{a_1}& \cdots & & \cdots & 0 \\
 \vdots   & \ddots&  & &  \vdots \\
   \vdots  & & \frac{1}{a_i} & &  \vdots \\
\vdots  & &   &  \ddots& \vdots\\
 0  &\cdots &  &\cdots  &  \frac{1}{a_m}\\
\end{array}
\right].
\end{equation*}
We multiply the positive matrix $A$ on both sides of the matrix $\Lambda$ 
\begin{eqnarray*}
     \overline{\Lambda} &=& A \Lambda A\\
     &=& I - \overline{L}^T \overline{L}.
\end{eqnarray*}
From linear algebra, we know  that the positivity of the matrix $\overline{\Lambda}$ is equivalent to 
\begin{equation*}
    1- \lvert \overline{L} \rvert^2 \geq 0.
\end{equation*}
Thus $\Lambda\ge 0$ is equivalent to the following inequality,
\begin{equation*}
    1-\sum^{m}_{i=1}\frac{b^2_i}{a^2_i}\geq 0.
\end{equation*}

\end{proof}
We also need the following algebraic lemmas which are proved in Yuan \cite{Yuan16Slag}.
\begin{lem} \label{kappa_iG_ii}
Suppose $\Theta = \frac{(n-2)\pi}{2}$, then we have
\begin{equation*}
    \sum^{n}_{i=1}\frac{\kappa_i}{1+\kappa_i^2}\ge 0.
\end{equation*} 
Moreover, if $n\geq3$ and $\lvert \kappa \rvert <C$, we have 
\begin{equation}
    \sum^{n}_{i=1}\frac{\kappa_i}{1+\kappa_i^2} > 0. \label{strictly}
\end{equation} 
\end{lem}
\begin{proof}
     There is at most one negative direction on critical phase say $\kappa_n<0$. Without loss of generality, we assume $\kappa_n<0$. Let $\theta_i=\arctan \kappa_i$. Then we have 
\begin{equation*}
    \pi > \pi + 2\theta_n = \sum\limits^{n-1}_{i=1}(\pi -2\theta_i)\geq0.
\end{equation*}
By elementary identities for $\sin$ function, we have 
\begin{eqnarray}
    \sin [\sum\limits^{n-1}_{i=1} (\pi - 2\theta_i)] &=&\sin(\pi -2 \theta_1) \cos  \sum\limits^{n-1}_{i=2} (\pi - 2\theta_i) + \sin \sum\limits^{n-1}_{i=2} (\pi - 2\theta_i)  \cos (\pi -2 \theta_1) \nonumber \\
    &\leq&  \sin (\pi - 2\theta_1) + \sin \sum\limits^{n-1}_{i=2} (\pi -2\theta_i) \label{theta_1} \\
    &\leq & \vdots \nonumber \\
    & \leq & \sum\limits^{n-1}_{i=1} \sin (\pi - 2\theta_i),\nonumber 
\end{eqnarray}
and 
\begin{eqnarray*}
    \sum\limits^{n-1}_{i=1} \sin (\pi- 2\theta_i) & =&  \sum\limits^{n-1}_{i=1} \sin (2\theta_i) .
\end{eqnarray*}
Thus, we have
    \begin{equation*}
    2\sum^{n}_{i=1}\frac{\kappa_i}{1+\kappa_i^2}=\sum^{n}_{i=1}\sin(2\theta_i)\ge \sin[\sum^{n-1}_{i=1}(\pi-2\theta_i)]-\sin(\pi+2\theta_n)=0.
\end{equation*} 
If $\lvert \kappa \rvert <C $, we know that 
\begin{equation*}
    \pi - 2 \theta_i >0, \quad for \quad \forall \quad i<n.
\end{equation*}
Therefore, the inequality \eqref{theta_1} becomes a strict inequality when $n\geq 3$, leading to \eqref{strictly}.

\end{proof}

The following inequality is employed in Yuan \cite{Yuan16Slag}.

\begin{lem} \label{inverseKappa}
When $\Theta \geq \frac{(n-2)\pi}{2}$ and $\kappa_n < 0$, we have 
\begin{equation*}
    \sum\limits^n_{i=1} \frac{1}{\kappa_i} \leq 0.
\end{equation*}
\end{lem}

\begin{proof}
Let us denote 
\begin{equation*}
    \theta_i =\arctan \kappa_i.
\end{equation*}
Thus our equation \eqref{eq:specialLag} is 
\begin{equation*}
    \sum\limits_i \theta_i =\Theta \geq \frac{(n-2)\pi}{2}.
\end{equation*}
When $\Theta\geq \frac{(n-2)\pi}{2}$ and $\kappa_n<0$, we have 
\begin{equation*}
    \frac{\pi}{2} > \frac{\pi}{2} +\theta_n \geq  \sum\limits^{n-1}_{i=1}(\frac{\pi}{2} -\theta_i)>0.
\end{equation*}
By an elementary identity for $\tan$ function, we have 
\begin{eqnarray*}
    \tan \sum\limits^{n-1}_{i=1}(\frac{\pi}{2} -\theta_i) &=&\frac{\tan (\frac{\pi}{2}-\theta_1) + \tan \sum\limits^{n-1}_{i=2}(\frac{\pi}{2} -\theta_i )}{1-\tan (\frac{\pi}{2}-\theta_1) \tan \sum\limits^{n-1}_{i=2}(\frac{\pi}{2} -\theta_i )}\\
    &\geq& \tan (\frac{\pi}{2}-\theta_1) + \tan \sum^{n-1}_{i=2}(\frac{\pi}{2} -\theta_i )\\
    & \geq& \vdots\\
    &\geq& \sum\limits^{n-1}_{i=1} \tan(\frac{\pi}{2}-\theta_i).
\end{eqnarray*}    
Thus 
\begin{equation*}
    -\frac{1}{\kappa_n}=\tan (\frac{\pi}{2} + \theta_n)  \geq \tan  \sum\limits^{n-1}_{i=1}(\frac{\pi}{2} -\theta_i)\geq \sum\limits^{n-1}_{i=1} \tan(\frac{\pi}{2}-\theta_i) = \sum\limits^{n-1}_{i=1} \frac{1}{\kappa_i}.
\end{equation*}
So we have 
\begin{equation*}
    \sum^n_{i=1}\frac1{\kappa_i}\le 0.
\end{equation*}   
\end{proof}

Let us consider the quantity of $b(x):=\log (H+J)$, where $J=J(n)$ is a given constant.
\begin{theorem}
\label{lem3} For admissible solutions of the equations (\ref{eq:specialLag})
in $\mathbb{R}^{n}$ with $\Theta \geq \frac{(n-2)\pi}{2}$ or $M$ is convex, we have
\begin{eqnarray}
\triangle_{G}b & \geq  \epsilon(n)\lvert\nabla_{G}b\rvert^{2}+\frac{|A|^2 G^{ij}h_{ij}-H G^{ij}h_{ik} \delta^{kl} h_{lj}}{J+H}.\label{eq:logb1}
\end{eqnarray}
Moreover, if $\Theta =\frac{(n-2)\pi}{2}$ or $M$ is convex, we have 
\begin{eqnarray}
\triangle_{G}b & \geq  \epsilon(n)\lvert\nabla_{G}b\rvert^{2}-n.\label{eq:logb}
\end{eqnarray}
\end{theorem}

\begin{proof}
In order to prove \eqref{eq:logb1} and \eqref{eq:logb}, we need to find a positive $\epsilon = \epsilon(n)$ such that the following quantity $Q(\epsilon)$ has a lower bound $-n (H+J)$:
\begin{eqnarray}
Q(\epsilon):&=& (H+J) ( \triangle_{G}b - \epsilon  \lvert\nabla_{G}b\rvert^{2} )\nonumber\\
& =&(H+J) G^{ij}b_{ij}- (H+J) \epsilon G^{ij}b_{i}b_{j} \nonumber \\
& = & G^{ij} H_{ij}-(1+\epsilon)\frac{G^{ij}H_{i}H_{j}}{H+J}. \label{eq:bii} 
\end{eqnarray}
We differentiate the equation \eqref{eq:specialLag}  once 
\begin{eqnarray}
G^{ij}h_{ijk}=\Theta_k=0.\label{G1}
\end{eqnarray}
Then, differentiating twice gives:
\begin{equation}
G^{ij}h_{ijkk}+\nabla_{k}G^{ij}h_{ijk} = \Theta_{kk}=0. \label{2ndG}
\end{equation}
Because 
\begin{eqnarray*}
\nabla_{k}G^{ij} & = & -G^{il}\nabla_{k}G_{lp}G^{pj}\\
 & = & -G^{il}(h_{lqk}\delta^{qq\prime} h_{q^\prime p}+h_{lq}\delta^{q q\prime} h_{q\prime pk})G^{pj}.
\end{eqnarray*}
Then, assuming at the point where $h_{ij}$ is diagonal and $\kappa_i = h_{ii}$, we have
\begin{equation*}
    G^{ij}= \frac{\delta_{ij}}{1+\kappa^2_i}
\end{equation*}
and 
\begin{eqnarray}
\nabla_{k}G^{ij}h_{ijk} & = & -G^{ii}(h_{ijk}h_{jj}+h_{ii}h_{ijk})G^{jj}h_{ijk} \nonumber\\
 & = & -G^{ii}G^{jj}(\kappa_{i}+\kappa_{j})h_{ijk}^{2}. \label{DeG}
\end{eqnarray}
Using \eqref{eq:commute2}, \eqref{2ndG} and \eqref{DeG}, the first part of (\ref{eq:bii}) is 
\begin{eqnarray*}
 G^{ij}H_{ij} & = &|A|^2 G^{ij}h_{ij}-H G^{ij}h_{ik} \delta^{kl} h_{lj} + \sum_k G^{ii}G^{jj}(\kappa_{i}+\kappa_{j})h_{ijk}^{2}.
\end{eqnarray*}
Thus we have
\begin{equation*}
    Q(\epsilon)=G^{ii}G^{jj}(\kappa_{i}+\kappa_{j})h_{ijk}^{2} - (1+\epsilon) G^{ii} \frac{(\sum\limits_k h_{kki})^2}{H+J}+|A|^2 G^{ij}h_{ij}-H G^{ij}h_{ik} \delta^{kl} h_{lj}.
\end{equation*}
Because $\Theta \geq \frac{(n-2)\pi}{2}$ or $M$ is convex, we have for $\forall$ $i\neq j$
\begin{equation*}
    \kappa_i+\kappa_j \geq 0.
\end{equation*}
Then, we have:
\begin{eqnarray*}
    Q(\epsilon)\geq & \sum\limits_{\gamma} \big[ \sum\limits_{i\neq \gamma} 2 (G^{ii})^2 \kappa_{i} h_{ii\gamma}^{2} + 2 (G^{\gamma\gamma})^2 \kappa_{\gamma} h_{\gamma \gamma \gamma}^{2}+ 2\sum\limits_{i\neq \gamma} G^{ii}G^{\gamma \gamma}(\kappa_{i}+\kappa_{\gamma})h_{i i \gamma}^{2} \\
    & - (1+\epsilon) G^{\gamma \gamma} \frac{(\sum\limits_i h_{ii\gamma})^2}{H+J} \big] +|A|^2 G^{ij}h_{ij}-H G^{ij}h_{ik} \delta^{kl} h_{lj}.
\end{eqnarray*}
The term involving $h_{ii\gamma}$ is denoted as follows: 
\begin{eqnarray*}
Q_{\gamma} & := & \sum_{i\neq \gamma}2[(G^{ii})^{2}\kappa_{i}+G^{\gamma\gamma}G^{ii}(\kappa_{\gamma}+\kappa_{i})]h_{ii\gamma}^{2}+2(G^{\gamma\gamma})^{2}\kappa_{\gamma}h_{\gamma\gamma\gamma}^{2}\\
 &  & -(1+\epsilon)G^{\gamma\gamma}\frac{(\sum\limits _{i} h_{ii\gamma})^{2}}{H+J}.
\end{eqnarray*}
Thus 
\begin{equation*}
    Q(\epsilon) \geq \sum\limits_\gamma Q_\gamma  +|A|^2 G^{ij}h_{ij}-H G^{ij}h_{ik} \delta^{kl} h_{lj}.
\end{equation*}
We will take two steps to estimate $Q(\epsilon)$. In the first step  we can deal with the case when $\Theta \geq \frac{(n-2)\pi}{2}$ or $M$ is convex.  And we will show that the quadratics $Q_\gamma$ are non-negative for each $1\le \gamma\le n$ if we properly choose $J,\epsilon$. In the second step we show that the remaining parts are bounded below by $-n(H+J)$ when $\Theta =\frac{(n-2)\pi}{2}$ or convex case.

Suppose $\kappa_1\geq \kappa_2 \geq \cdots \geq \kappa_n$.
Denote 
\begin{equation*}
    \tau^j_i := \frac{G^{jj}}{G^{ii}},
\end{equation*}
we can infer from the equation for any $\gamma$ from $1$ to $n$ that
\begin{equation}
    h_{nn\gamma} =-\sum^{n-1}_{i=1} \tau^i_n h_{ii\gamma}.\label{h_nngamma}
\end{equation}
\textbf{Case 1.} $\kappa_n<0$. 

\textbf{When }  $\gamma=n$. 
Replacing the expression $h_{nnn}$ in  \eqref{h_nngamma} into $Q_n$, we obtain 
\begin{eqnarray*}
  Q_{n} & = & \sum^{n-1}_{i=1} 2[(G^{ii})^{2}\kappa_{i}+G^{nn}G^{ii}(\kappa_{n}+\kappa_{i})]h_{iin}^{2}+2\kappa_{n}(\sum^{n-1}_{i=1} G^{ii} h_{iin})^{2}\\
 &  & -(1+\epsilon)G^{nn}\frac{[\sum\limits ^{n-1}_{i=1} (1-\tau^i_n)h_{iin}]^{2}}{H+J}.
\end{eqnarray*}
Then we split $Q_n$ into the following two parts:
\begin{equation*}
    Q^{(1)}_n := 2 \sum^{n-1}_{i=1} \kappa_i (G^{ii} h_{iin})^2 + 2 \kappa_{n}(\sum^{n-1}_{i=1} G^{ii} h_{iin})^{2},
\end{equation*}
and
\begin{equation*}
    Q^{(2)}_n := 2 \sum^{n-1}_{i=1} G^{nn} G^{ii}(\kappa_n+\kappa_i) h_{iin}^2 -(1+\epsilon)G^{nn}\frac{[\sum\limits ^{n-1}_{i=1} (1-\tau^i_n)h_{iin}]^{2}}{H+J}.
\end{equation*}
In order to prove $Q^{(1)}_n\geq 0$, we apply the diagonalization Lemma \ref{Diag} by choosing $x_i = G^{ii}h_{iin}$, $a_i = \sqrt{2 \kappa_i}$ and $b_i= \sqrt{-2\kappa_n}$.
Thus $Q^{(1)}_n\geq 0$ is equivalent to 
\begin{equation}
    1+ \sum\limits^{n-1}_{i=1} \frac{\kappa_n}{\kappa_i} \geq 0.\label{Q_n^1}
\end{equation}
The inequality \eqref{Q_n^1} is from Lemma \ref{inverseKappa}.

Next, our goal is to prove $Q^{(2)}_n\ge0$ for some small $\epsilon$. We apply the diagonalization Lemma \ref{Diag} by choosing the following
\begin{eqnarray*}
     x_i & = &h_{iin}\\
     a_i&=&\sqrt{2G^{nn}G^{ii}(\kappa_n+\kappa_i)}\\
     b_i &=& \frac{\sqrt{(1+\epsilon)G^{nn} } (1-\tau^i_n)}{\sqrt{H+J}}.
\end{eqnarray*}
Thus we only need to study the positivity of the following quantity
\begin{eqnarray*}
    P_n &:=& 1-(1+\epsilon)\sum^{n-1}_{i=1}\frac{(1-\tau^i_{n})^2}{2(\kappa_i+\kappa_n)G^{ii}(H+J)} \\
    &\geq&  1-(1+\epsilon)\frac{1}{2H}\sum^{n-1}_{i=1} \frac{(\kappa_i-\kappa_n)^2 (\kappa_i+\kappa_n)}{(1+\kappa^2_i)}\\
    &\geq & 1-(1+\epsilon)\frac1{2H}\sum^{n-1}_{i=1}(\kappa_i-\kappa_n)(1-\frac{\kappa_n^2}{\kappa_i^2}).
\end{eqnarray*} 
Let $\kappa_i=(1+t_i)|\kappa_n|$, then we have
\begin{eqnarray}
    P_n &\geq& 1-(1+\epsilon)\frac{\sum\limits^{n-1}_{i=1} (2+t_i)(1-\frac{1}{(1+t_i)^2})}{2(n-2+\sum\limits^{n-1}_{i=1}t_i)} \nonumber\\
    & = & 1 - (1+\epsilon) \frac{1}{2}[1+\frac{n}{n-2+\sum\limits^{n-1}_{i=1}t_i}- \frac{\sum\limits^{n-1}_{i=1} \frac{2+t_i}{(1+t_i)^2}}{n-2+\sum\limits^{n-1}_{i=1}t_i}]\nonumber\\
     &\geq&  1 - (1+\epsilon) \frac{1}{2}[1+\frac{n-\sum\limits^{n-1}_{i=1} \frac{1}{1+t_{i}}}{n-2+\sum\limits^{n-1}_{i=1}t_i}] \label{n=3.1}\\ 
 & \geq & 1 - (1+\epsilon) \frac{1}{2}[1+\frac{n}{n-2+\sum\limits^{n-1}_{i=1}t_i}]. \label{n=3.2} 
\end{eqnarray}
From $\sum^n_{i=1} \frac{1}{\kappa_i} \leq 0 $, we have $\sum\limits^{n-1}_{i=1}\frac{1}{1+t_i} \leq 1$. \\
Then we claim that
\begin{equation}
    \sum\limits^{n-1}_{i=1} t_i \geq (n-1)(n-2).\label{sumt}
\end{equation}
In fact, by Cauchy-Schwarz inequality we have
\begin{eqnarray*}
    (n-1)^2 &\leq& [\sum\limits^{n-1}_{i=1}(1+t_i)][\sum\limits^{n-1}_{i=1}\frac{1}{1+t_i}]\\
    &\leq & [\sum\limits^{n-1}_{i=1}(1+t_i)]\\
    & = & (n-1)+\sum\limits^{n-1}_{i=1}t_i.
\end{eqnarray*}
If $\sum\limits^{n-1}_{i=1} \frac{1}{1+t_i} \leq \frac{1}{2}$, we have a better estimate
\begin{equation}
    \sum\limits^{n-1}_{i=1} t_i \geq (n-1)(2n-3). \label{sumt_0.5}
\end{equation}
If $n\geq4$, we have from \eqref{n=3.2} and \eqref{sumt}
\begin{eqnarray}
 P_n & \geq & 1 - (1+\epsilon) \frac{1}{2}[1+\frac{n}{n-2+(n-1)(n-2)}]\nonumber\\
 &=& 1-(1+\epsilon) \frac{n-1}{2(n-2)}\nonumber\\
 &\geq & 1 - (1+\epsilon) \frac{3}{4}.\nonumber
\end{eqnarray}
If $n=3$ and $\sum\limits^{n-1}_{i=1} \frac{1}{1+t_i} \geq \frac{1}{2}$, we have from \eqref{n=3.1} and \eqref{sumt} that
\begin{eqnarray*}
      P_n &\geq& 1  - (1+\epsilon) \frac{1}{2}[1+\frac{n-\frac{1}{2}}{n-2+(n-1)(n-2)}]\\
      &=& 1- (1+\epsilon) \frac{1}{2} [1+\frac{\frac{5}{2}}{1+2}]\\
      &=& 1- (1+\epsilon) \frac{11}{12}.
\end{eqnarray*}
If $n=3$ and $\sum\limits^{n-1}_{i=1} \frac{1}{1+t_i} \leq \frac{1}{2}$, we have from \eqref{n=3.2} and \eqref{sumt_0.5} that
\begin{eqnarray*}
    P_n &\geq& 1-\frac{1+\epsilon}{2} [1+\frac{n}{n-2 +(n-1)(2n-3)}]\\
    &=& 1-\frac{1+\epsilon}{2}[1+\frac{3}{1+6}]\\
    &=& 1-\frac{5}{7}(1+\epsilon).
\end{eqnarray*}

\textbf{When} $\gamma\le n-2$. 

Let $0<\hat \epsilon<1$. We replace the expression $h_{nn\gamma}$ in  \eqref{h_nngamma} into $Q_\gamma$ and split $Q_\gamma$ into two quadratic forms,
\begin{eqnarray*}
    Q^{(1)}_\gamma &:=& (1-\hat \epsilon) \kappa_\gamma (G^{\gamma \gamma})^2 h^2_{\gamma \gamma \gamma} + 2\sum^{n-1}_{i\neq \gamma} (\kappa_i+\tau^\gamma_i \kappa_\gamma) (G^{ii})^2 h^2_{ii\gamma} + 2[\kappa_n+\tau^\gamma_n (\kappa_\gamma + \kappa_n)] (\sum\limits^{n-1}_{i=1} G^{ii}h_{ii\gamma})^2,\\
    Q^{(2)}_\gamma &:=& (1+\hat{\epsilon}) \kappa_\gamma (G^{\gamma \gamma})^2h^2_{\gamma\gamma\gamma} + 2\sum\limits^{n-1}_{i\neq \gamma} \kappa_i G^{ii} G^{\gamma \gamma} h^2_{ii\gamma}-(1+\epsilon) \frac{G^{\gamma\gamma}[\sum\limits^{n-1}_{i=1} (1-\tau^i_n)h_{ii\gamma}]^2}{H+J}.
\end{eqnarray*}
 We need to determine $\hat \epsilon<1$, $\epsilon$ such that $Q^{(1)}_\gamma \geq 0$ and $Q^{(2)}_\gamma \geq 0$. 

 If 
\begin{equation*}
    \kappa_n+(\kappa_\gamma+\kappa_n)\tau_n^\gamma\geq 0,
\end{equation*} 
it is easy to see that 
\begin{equation*}
  Q^{(1)}_\gamma \geq 0.   
\end{equation*}
So we  assume that $\kappa_n+(\kappa_\gamma+\kappa_n)\tau^\gamma
_{n}<0$. 
 Let $x_i= h_{ii\gamma}$ and we choose 
 \begin{eqnarray*}
     a^{(1)}_\gamma &=& \sqrt{(1-\hat{\epsilon})\kappa_\gamma} G^{\gamma \gamma} \\
     a^{(1)}_i &=& \sqrt{2(\kappa_i+\tau^\gamma_i \kappa_\gamma)} G^{ii} \quad for \quad i\neq \gamma \\
     b^{(1)}_i &=& \sqrt{-2[\kappa_n+\tau^\gamma_n(\kappa_\gamma+\kappa_n)]} G^{ii},
 \end{eqnarray*}
 and 
 \begin{eqnarray*}
     a^{(2)}_\gamma &=& \sqrt{(1+\hat{\epsilon})\kappa_\gamma} G^{\gamma \gamma} \\
     a^{(2)}_i &=& \sqrt{2 \kappa_i G^{ii} G^{\gamma \gamma}}  \quad for \quad i\neq \gamma \\
     b^{(2)}_i &=& \frac{\sqrt{(1+\epsilon) G^{\gamma\gamma}}(1-\tau^i_n)}{\sqrt{H+J}}.
 \end{eqnarray*}
 Thus by the diagonalization Lemma \ref{Diag}, $Q^{(1)}_\gamma \geq 0$ is equivalent to the following inequality holds,
 \begin{eqnarray*}
     P_1:= 1+ \frac{2[\kappa_n+\tau^\gamma_n(\kappa_\gamma+\kappa_n)]}{(1-\hat{\epsilon} )\kappa_\gamma} + \sum\limits^{n-1}_{i\neq \gamma} \frac{\kappa_n+\tau^\gamma_n(\kappa_\gamma+\kappa_n)}{\kappa_i +\tau^\gamma_i \kappa_\gamma} \geq 0.
 \end{eqnarray*}
Let $\kappa_i=(1+t_i)|\kappa_n|$, then we have
 \begin{eqnarray*}
     \frac{P_1}{1-\tau^\gamma_n t_\gamma} &=& \frac{1}{1-\tau^\gamma_n t_\gamma}- 2\frac{1}{(1-\hat{\epsilon})(1+t_\gamma)} - \sum\limits^{n-1}_{i\neq\gamma} \frac{1}{1+t_i + \tau^\gamma_i (1+t_\gamma)}\\
     &= & 1  +(\frac{1}{1-\tau^\gamma_n t_\gamma}- 1)-\frac{1}{1+t_\gamma} +[\frac{1}{1+t_\gamma} -\frac{2}{(1-\hat{\epsilon})(1+t_\gamma)}] \\
     & &-\sum\limits^{n-1}_{i\neq\gamma} \frac{1}{1+t_i} +\sum\limits^{n-1}_{i\neq\gamma} [\frac{1}{1+t_i}-\frac{1}{1+t_i+\tau^\gamma_i(1+t_\gamma)}]\\
     &\geq& 1 -\frac{1}{1+t_\gamma}  -\sum\limits^{n-1}_{i\neq\gamma} \frac{1}{1+t_i}\\
     & &+\frac{\tau^\gamma_n t_\gamma}{1-\tau^\gamma_n t_\gamma}-\frac{1+\hat{\epsilon}}{(1-\hat{\epsilon})(1+t_\gamma)}
     +\frac{\tau^\gamma_{n-1}(1+t_\gamma)}{(1+t_{n-1})[1+t_{n-1}+\tau^\gamma_{n-1}(1+t_\gamma)]}.
\end{eqnarray*}
We observe that when $\gamma \leq n-2$
\begin{equation}
    \frac{\tau^\gamma_{n-1}(1+t_\gamma)}{(1+t_{n-1})[1+t_{n-1}+\tau^\gamma_{n-1}(1+t_\gamma)]} \geq \frac{1}{1+t_{n-1} + 1+t_\gamma}.\label{OberP_1}
\end{equation}
Due to 
\begin{eqnarray*}
    1 &\geq& \sum\limits^{n-1}_{i=1} \frac{1}{1+t_i} \\
    \tau^\gamma_{n} & \geq& \frac{\kappa^2_{n}}{\kappa^2_\gamma}=\frac{1}{(1+t_\gamma)^2}
\end{eqnarray*}
and \eqref{OberP_1}, we have 
\begin{eqnarray*}
    \frac{P_1 (1+t_\gamma)}{1-\tau^\gamma_n t_\gamma} &\geq&  \frac{t_\gamma}{(1+t_\gamma) [1-\frac{t_\gamma}{(1+t_\gamma)^2}]} - \frac{1+\hat{\epsilon}}{1-\hat{\epsilon}} +\frac{1}{\frac{1+t_{n-1}}{1+t_\gamma} +1} \\
    &\geq& \frac{t_\gamma (1+t_\gamma)}{1+t_\gamma+t^2_\gamma} - \frac{1+\hat{\epsilon}}{1-\hat{\epsilon}} + \frac{1}{2}.
\end{eqnarray*}
If $n\geq3$, we know from \eqref{eq:kappan-1} that  $t_\gamma \geq 1$. Thus we have 
\begin{equation*}
    \frac{t_\gamma (1+t_\gamma)}{1+t_\gamma+t^2_\gamma} \geq \frac{2}{3}.
\end{equation*}
Then \begin{eqnarray*}
    \frac{P_1 (1+t_\gamma)}{1-\tau^\gamma_n t_\gamma} 
    &\geq & \frac{7}{6} - \frac{1+\hat{\epsilon}}{1-\hat{\epsilon}}.
\end{eqnarray*}
 So we can choose $\hat{\epsilon}\leq \frac{1}{13}$ such that $P_1\geq0$.\\
 And $Q^{(2)}_\gamma \geq 0$ is equivalent to the following inequality holds,
  \begin{eqnarray*}
    P_2: = 1- \frac{(1+\epsilon)(1-\tau^\gamma_n)^2}{(H+J)(1+\hat{\epsilon})\kappa_\gamma G^{\gamma\gamma}} - \sum\limits^{n-1}_{i\neq \gamma} \frac{(1+\epsilon)(1-\tau^i_n)^2}{2(H+J)\kappa_i G^{ii}} \geq 0.
 \end{eqnarray*}
Because for any $i<n$, we have
\begin{equation}
    \frac{(1-\tau_n^i)^2 }{G^{ii} \kappa_i} \leq \kappa_i (1-\frac{\kappa^2_n}{\kappa^2_i}).\label{Kn/ki}
\end{equation}
We have from \eqref{Kn/ki}
\begin{eqnarray*}
    P_2 &\geq& 1-(1+\epsilon)\frac{1}{H}[\frac1{1+\hat\epsilon}\frac{(1-\tau^{\gamma}_{n})^2G_{\gamma\gamma}}{\kappa_\gamma }
+\sum^{n-1}_{i\neq\gamma}\frac{(1-\tau^i_{n})^2G_{ii}}{2\kappa_i }] \\
&\geq &1-(1+\epsilon)\frac1{H}[\frac1{1+\hat\epsilon}\kappa_\gamma(1-\frac{\kappa_n^2}{\kappa_\gamma^2})
+\frac12\sum^{n-1}_{i\neq\gamma}\kappa_i(1-\frac{\kappa_n^2}{\kappa_i^2})].
\end{eqnarray*}
Let $\kappa_i=(1+t_i)|\kappa_n|$, then we have
\begin{eqnarray*}
P_2 &\geq &1-\frac{1+\epsilon}{n-2+\sum^{n-1}_{i=1}t_i}[\frac{1+t_\gamma}{1+\hat\epsilon}(1-\frac1{(1+t_\gamma)^2})
+\sum^{n-1}_{i\neq\gamma}\frac{1+t_i}2(1-\frac1{(1+t_i)^2})]\\
& \geq & 1-\frac{1+\epsilon}{n-2+\sum^{n-1}_{i=1}t_i}[\frac{1+t_\gamma}{1+\hat\epsilon}
+\sum^{n-1}_{i\neq\gamma}\frac{1+t_i}{2} -\frac{1}{2(1+t_{n-1})}]\\
&=& 1 -\frac{1+\epsilon}{1+\hat{\epsilon}} \frac{2(1+t_\gamma)+(1+\hat{\epsilon})[n-2+\sum\limits^{n-1}_{i\neq \gamma} t_i-\frac{1}{1+t_{n-1}}]}{2(n-2+\sum\limits^{n-1}_{i=1} t_i)}.\\
\end{eqnarray*}
Because
\begin{equation*}
    \frac{1}{1+t_{n-1}}  \geq 1-t_{n-1},
\end{equation*}
we obtain
\begin{eqnarray*}
P_2 &\geq & 1 -\frac{1+\epsilon}{1+\hat{\epsilon}} \frac{2(1+t_\gamma)+(1+\hat{\epsilon})[n-2+\sum\limits^{n-1}_{i\neq \gamma} t_i-1+t_{n-1}]}{2(n-2+\sum\limits^{n-1}_{i=1} t_i)}.\\
\end{eqnarray*}
If $n\geq 3$ and $\gamma \leq n-2$, we observe that 
\begin{equation*}
    \frac{1}{2+\hat{\epsilon}} \frac{2(1+t_\gamma)+(1+\hat{\epsilon})[n-2+\sum\limits^{n-1}_{i\neq \gamma} t_i -1+t_{n-1}]}{n-2+\sum\limits^{n-1}_{i=1} t_i} <1.
\end{equation*}
Thus if $\hat\epsilon=4\epsilon \leq 2$, we have
\begin{equation*}
    P_2 \geq 1 -\frac{(1+\epsilon)(1+\frac{\hat{\epsilon}}{2})}{1+\hat{\epsilon}}\geq 0.
\end{equation*}


\textbf{When} $\gamma=n-1$.\\

Replacing the expression $h_{nn\gamma}$ in  \eqref{h_nngamma} into $Q_\gamma$, we split $Q_\gamma$ into two parts and consider the following two quadratic forms with new $\hat \epsilon$,
\begin{eqnarray*}
    \hat{Q}^{(1)}_\gamma &:=& (1+\hat \epsilon) \kappa_\gamma (G^{\gamma \gamma})^2 h^2_{\gamma \gamma \gamma} + 2\sum^{n-1}_{i\neq \gamma} (\kappa_i+\tau^\gamma_i \kappa_\gamma) (G^{ii})^2 h^2_{ii\gamma} + 2[\kappa_n+\tau^\gamma_n (\kappa_\gamma + \kappa_n)] (\sum\limits^{n-1}_{i=1} G^{ii}h_{ii\gamma})^2,\\
    \hat{Q}^{(2)}_\gamma &:=& (1-\hat{\epsilon}) \kappa_\gamma (G^{\gamma \gamma})^2h^2_{\gamma\gamma\gamma} + 2\sum\limits^{n-1}_{i\neq \gamma} \kappa_i G^{ii} G^{\gamma \gamma} h^2_{ii\gamma}-(1+\epsilon) \frac{G^{\gamma\gamma}[\sum\limits^{n-1}_{i=1} (1-\tau^i_n)h_{ii\gamma}]^2}{H+J}.
\end{eqnarray*}
Without loss of generality, we also assume that $\kappa_n+(\kappa_\gamma+\kappa_n)\tau^\gamma_{n}<0$.
Let $x_i= h_{ii\gamma}$ and we choose 
 \begin{eqnarray*}
     a^{(1)}_\gamma &=& \sqrt{(1+\hat{\epsilon})\kappa_\gamma} G^{\gamma \gamma} \\
     a^{(1)}_i &=& \sqrt{2(\kappa_i+\tau^\gamma_i \kappa_\gamma)} G^{ii} \quad for \quad i\neq \gamma \\
     b^{(1)}_i &=& \sqrt{-2[\kappa_n+\tau^\gamma_n(\kappa_\gamma+\kappa_n)]} G^{ii},
 \end{eqnarray*}
 and 
 \begin{eqnarray*}
     a^{(2)}_\gamma &=& \sqrt{(1-\hat{\epsilon})\kappa_\gamma} G^{\gamma \gamma} \\
     a^{(2)}_i &=& \sqrt{2 \kappa_i G^{ii} G^{\gamma \gamma}}  \quad for \quad i\neq \gamma \\
     b^{(2)}_i &=& \frac{\sqrt{(1+\epsilon) G^{\gamma\gamma}}(1-\tau^i_n)}{\sqrt{H+J}}.
 \end{eqnarray*}
Thus by the diagonalization Lemma \ref{Diag},  $\hat Q^{(1)}_\gamma\ge 0,\hat Q^{(2)}_\gamma\ge 0$ is equivalent to the following two inequalities hold
 \begin{eqnarray*}
    \hat{P}_1:= 1+ \frac{2[\kappa_n+\tau^\gamma_n(\kappa_\gamma+\kappa_n)]}{(1+\hat{\epsilon} )\kappa_\gamma} + \sum\limits^{n-2}_{i= 1} \frac{\kappa_n+\tau^\gamma_n(\kappa_\gamma+\kappa_n)}{\kappa_i +\tau^\gamma_i \kappa_\gamma} \geq 0
 \end{eqnarray*}
 and 
  \begin{eqnarray*}
     \hat{P}_2 := 1- \frac{(1+\epsilon)(1-\tau^\gamma_n)^2}{(H+J)(1-\hat{\epsilon})\kappa_\gamma G^{\gamma\gamma}} - \sum\limits^{n-2}_{i=1} \frac{(1+\epsilon)(1-\tau^i_n)^2}{2(H+J)\kappa_i G^{ii}} \geq 0.
 \end{eqnarray*}
By Lemma \ref{Diag}, $\hat{P}_1 \geq 0$ is equivalent to the following inequality:
 \begin{eqnarray*}
 \frac{\hat{P}_1}{\kappa_n+\tau^\gamma_n(\kappa_\gamma + \kappa_n)} & = &   \frac{1}{\kappa_n+\tau^\gamma_n(\kappa_\gamma + \kappa_n)}+ \frac{2}{(1+\hat{\epsilon})\kappa_\gamma}+\sum\limits^{n-2}_{i=1} \frac{1}{\kappa_i + \tau^\gamma_i \kappa_\gamma} \\
 &=& \sum^n_{i=1} \frac{1}{\kappa_i}
  -\frac{\tau^\gamma_n(\kappa_\gamma +\kappa_n)}{[\kappa_n +\tau^\gamma_n (\kappa_\gamma +\kappa_n)]\kappa_n} +\frac{1-\hat{\epsilon}}{(1+\hat{\epsilon})\kappa_\gamma} -\sum\limits^{n-2}_{i=1} \frac{\tau^\gamma_i \kappa_\gamma}{\kappa_i(\kappa_i + \tau^\gamma_i \kappa_\gamma)} \leq 0.
 \end{eqnarray*}
 Due to $ \sum\limits^{n}_{i=1} \frac{1}{\kappa_i} \leq 0$, we have
 \begin{eqnarray*}
    \frac{\kappa_\gamma \hat{P}_1}{\kappa_n+\tau^\gamma_n(\kappa_\gamma + \kappa_n)} & \leq & -\frac{\tau^\gamma_n (\kappa_\gamma + \kappa_n)\kappa_\gamma}{[\kappa_n+\tau^\gamma_n(\kappa_\gamma+\kappa_n)]\kappa_n} + \frac{1-\hat{\epsilon}}{1+\hat{\epsilon}}\\
    & \leq & -\frac{\kappa^2_n (\kappa_\gamma+\kappa_n)}{\kappa_\gamma [\kappa_n + \frac{\kappa^2_n}{\kappa^2_\gamma}(\kappa_\gamma+\kappa_n)]\kappa_n} +\frac{1-\hat{\epsilon}}{1+\hat{\epsilon}}.
 \end{eqnarray*}
 The last inequality holds because $\tau^\gamma_n \geq \frac{\kappa^2_n}{\kappa^2_\gamma}$.
 Thus we choose $\hat \epsilon=\frac{\kappa_n^2}{\kappa_\gamma^2}$,
 \begin{eqnarray*}
      \frac{\kappa_\gamma \hat{P}_1}{\kappa_n+\tau^\gamma_n(\kappa_\gamma + \kappa_n)} 
     &\leq & -\frac{ 1+\frac{\kappa_n}{\kappa_\gamma}}{ 1 + \frac{\kappa_n}{\kappa_\gamma}(1+\frac{\kappa_n}{\kappa_\gamma})} +\frac{1-\frac{\kappa^2_n}{\kappa^2_\gamma}}{1+\frac{\kappa^2_n}{\kappa^2_\gamma}}\\
     &\leq & 0.
 \end{eqnarray*}
By \eqref{Kn/ki}, we have 
\begin{eqnarray*}
    \hat{P}_2 &\geq& 1-(1+\epsilon)\frac{1}{H}[\frac1{1-\hat\epsilon}\frac{(1-\tau^\gamma_{n})^2G_{\gamma\gamma}}{\kappa_\gamma }+ \frac1{2}\sum^{n-2}_{i=1}\frac{(1-\tau^i_{n})^2G_{ii}}{\kappa_i}]\\
    & \geq & 1- \frac{1+\epsilon}{H} [\frac{1}{1-\hat{\epsilon}}\kappa_\gamma (1-\frac{\kappa^2_n}{\kappa^2_\gamma}) +\frac{1}{2}\sum\limits^{n-2}_{i=1} \kappa_i (1-\frac{\kappa^2_n}{\kappa^2_i})].
\end{eqnarray*}
Let $\kappa_i=(1+t_i)|\kappa_n|$, $\hat \epsilon=\frac{\kappa_n^2}{\kappa_\gamma^2}$ and $\gamma =n-1$, then we have
\begin{eqnarray*}
    \hat{P}_2 &\geq & 1 -\frac{1+\epsilon}{n-2+\sum\limits^{n-1}_{i=1} t_i} [1+t_{n-1} + \frac{1}{2} \sum\limits^{n-2}_{i=1} (1+t_i-\frac{1}{1+t_i})] .
\end{eqnarray*}
If $n\geq4$, we observe that 
\begin{equation*}
    \frac{1}{n-2+\sum\limits^{n-1}_{i=1} t_i} [1+t_{n-1}+ \frac{1}{2} \sum\limits^{n-2}_{i=1} (1+t_i)] \leq \frac{3}{4}.
\end{equation*}
Thus we have
\begin{equation*}
    \hat{P}_2 \geq 1-\frac{3(1+\epsilon)}{4}.
\end{equation*}
If $n=3$ and $t_1 \leq 2$, we have 
\begin{eqnarray*}
    \hat{P}_2 &\geq & 1 -\frac{1+\epsilon}{1+ t_1+t_2} [1+t_2 + \frac{1}{2}  (1+t_1-\frac{1}{1+t_1})] \\
    & \geq & 1 -\frac{1+\epsilon}{1+ t_1+t_2} [1+t_2 + \frac{1}{2}  (t_1+\frac{2}{3})].
\end{eqnarray*}
Because $t_1\geq \max \{ t_2,1 \}$, we have that 
\begin{equation*}
    \frac{1}{1+t_1+t_2} [1+t_2 + \frac{1}{2} (t_1+\frac{2}{3})] \leq \frac{17}{18}.
\end{equation*}
Thus 
\begin{eqnarray*}
    \hat{P}_2  & \geq & 1 -\frac{17}{18}(1+\epsilon).
\end{eqnarray*}
If $n=3$ and $t_1 \geq 2$, we have that
\begin{equation*}
    \frac{1}{1+t_1+t_2} [1+t_2 + \frac{1}{2} (t_1+1)] \leq \frac{9}{10}.
\end{equation*}
Thus we obtain
\begin{eqnarray*}
    \hat{P}_2  & \geq & 1 -\frac{9}{10}(1+\epsilon).
\end{eqnarray*}
\textbf{Case 2.} $\kappa_n\ge0$. \\

We consider the quadratic form
$$
Q_\gamma=\sum^{n}_{i=1}
2\kappa_i(J+H) G^{ii}h_{ii\gamma}^2
-(1+\epsilon)(\sum^{n}_{i=1}h_{ii\gamma})^2.
$$
If $\kappa_{n}\ge 1$, to verify $Q_\gamma\ge0$ we only need $\epsilon \leq 1$ and $J\geq n$ such that
\begin{align*}
&1-(1+\epsilon)\sum^{n}_{i=1}\frac{G_{ii}}{2\kappa_i(J+H)}\ge 1-\frac{1+\epsilon}{2(J+H)}\sum^{n}_{i=1}(\kappa_i+1 )\ge 1-\frac{1+\epsilon}2\ge0.
\end{align*}
If $\kappa_{1}\leq 1$, we verify directly that for $J\geq n$. Due to the equation $G^{ii}h_{ii\gamma} =0$, we have 
\begin{align*}
&(\sum^{n}_{i=1}h_{ii\gamma})^2=[\sum^{n}_{i=1}(h_{ii\gamma}-G^{ii}h_{ii\gamma})]^2\le n  \sum^{n}_{i=1}\kappa_iG^{ii}h^2_{ii\gamma}
\le \sum^{n}_{i=1}
\kappa_i(J+H) G^{ii}h_{ii\gamma}^2.
\end{align*}
In this case, we also have $Q_\gamma \geq 0$. \\
Otherwise $\kappa_n< 1$ and $\kappa_1> 1$, we split $\{\kappa_i\}^n_{i=1}$ into two parts,
$$
\left\{
\begin{array}{ll}
 (\kappa_1,\cdots,\kappa_k)    & \kappa_k\ge 1   \\
(\kappa_{k+1},\cdots ,\kappa_{n})   &  \kappa_{k+1}< 1     \\
\end{array} \right.
$$
where $k\le n-1$. \\
By the above assumption, we have 
\begin{eqnarray*}
\sum^k_{i=1}\frac12\kappa_i(J+H)G^{ii}h^2_{ii\gamma} &\geq& \sum^k_{i=1}\frac{\kappa_i J}{2(1+\kappa^2_i)}h^2_{ii\gamma} \\
&\geq & J \sum\limits^k_{i=1} \frac{h^2_{ii\gamma}}{(1+\kappa^2_i)^2}. 
\end{eqnarray*}
Then by Cauchy-Schwarz inequality and our equation $G^{ii}h_{ii\gamma} =0$, we have
\begin{eqnarray*}
\sum^k_{i=1}\frac{1}{2}\kappa_i(J+H)G^{ii}h^2_{ii\gamma}
&\geq & \frac{J}{n} (\sum\limits^k_{i=1} \frac{h_{ii\gamma}}{1+\kappa^2_i})^2 \\
& =&\frac{J}{n} (\sum\limits^n_{i=k+1} \frac{h_{ii\gamma}}{1+\kappa^2_i})^2\\
&\geq& \frac{J}{4n} (\sum\limits^n_{i=k+1} h_{ii\gamma})^2 . 
\end{eqnarray*}
Thus if we choose $J\geq 4n^3$, we have
\begin{equation*} \sum^k_{i=1}\frac{1}{2}\kappa_i(J+H)G^{ii}h^2_{ii\gamma} \geq n^2 (\sum\limits^n_{i=k+1} h_{ii\gamma})^2.
\end{equation*}
We get
\begin{equation*}
   Q_\gamma \ge \sum^{k}_{i=1}
\frac32\kappa_i(J+H) G^{ii}h_{ii\gamma}^2 +n^2(\sum^n_{i=k+1}h_{ii\gamma})^2 -(1+\epsilon)(\sum^n_{i=1}h_{ii\gamma})^2.
\end{equation*}
Now by the diagonalization lemma \ref{Diag} with 
\begin{eqnarray*}
    x_{k+1} & =& 
    \sum^n_{i=k+1}h_{ii\gamma}\\
    a_{k+1} &=& n\\
    x_i & = & h_{ii\gamma}  \quad i\leq k \\
    a_i &=& \sqrt{\frac{3\kappa_i (J+H)G^{ii}}{2}} \quad i\leq k\\
    b_i & = & \sqrt{1+\epsilon} \quad i\leq k+1,
\end{eqnarray*}
we need to estimate as follows:
\begin{align*}
1-(1+\epsilon)(\sum^{k}_{i=1}\frac{2G_{ii}}{3\kappa_i(J+H)}+\frac1{n^2})&\ge 1-\frac23\frac{1+\epsilon}{(J+H)}\sum^{k}_{i=1}(\kappa_i+1 ) - \frac{1+\epsilon}{n^2}\\
&\ge 1-\frac{11}{12}(1+\epsilon)\ge0.
\end{align*}
In all, if we choose $\epsilon \leq \frac{1}{17}$ and $J\geq 4 n^3$ we finishes the first step to get
\begin{equation*}  Q(\epsilon)\ge\sum^n_{\gamma=1}\kappa_\gamma^2\sum^n_{i=1}G^{ii}\kappa_i-H\sum^n_{i=1}G^{ii}\kappa_i^2.
\end{equation*}

When $\Theta =\frac{(n-2)\pi}{2}$, we know from Lemma \ref{kappa_iG_ii} that
\begin{equation*}
    \sum^n_{i=1}\kappa_iG^{ii}\ge 0.
\end{equation*}
For the convex case, it is obvious that $\sum^n_{i=1}\kappa_iG^{ii}\ge 0$. Then, either $\Theta = \frac{(n-2)\pi}{2}$, or $M$ is convex, we have 
\begin{equation*}
Q(\epsilon)\ge\sum^n_{\gamma=1}\kappa_\gamma^2\sum^n_{i=1}G^{ii}\kappa_i-H\sum^n_{i=1}G^{ii}\kappa_i^2\ge -nH.
\end{equation*}

\end{proof}

\section{Mean value inequality.}
Because the graph 
\begin{equation*}
    X^\Sigma=(X,\nu)=(x_{1},\cdots,x_{n},u,\frac{u_{1}}{W},\cdots,\frac{u_{n}}{W},-\frac{1}{W})
\end{equation*}
where $u$ satisfied equation (\ref{eq:specialLag2}) can be viewed
as a $n$ dimensional smooth submanifold in $(\mathbb{R}^{n+1}\times\mathbb{S}^{n},\sum\limits _{i=1}^{n+1}dx_{i}^{2}+i_{\mathbb{S}^{n}}^{*}(\sum\limits _{i=1}^{n+1}dy_{i}^{2}))$
here $i:\mathbb{S}^{n}\hookrightarrow\mathbb{R}^{n+1}$ is the standard
embedding. We observe that it is a submanifold with a bounded
mean curvature. 
In fact, we have 
\begin{eqnarray*}
X_{i}^{\Sigma} & = & (X_{i},\nu_{i})=(X_{i},h_{i}^{k}X_{k}),
\end{eqnarray*}
and 
\[
G_{ij}=<X_{i}^{\Sigma},X_{j}^{\Sigma}>_{\mathbb{R}^{n+1}\times\mathbb{S}^{n}}=g_{ij}+h_{i}^{k}h_{kj}.
\]

\begin{eqnarray*}
\lvert\mathscr{H}_{1}\rvert&=&\lvert<G^{ij}(D_{j}X_{i}^{\Sigma}- 
(\Gamma_{ji}^{k})^{\Sigma}X_{k}^{\Sigma}),\nu_{1}^{\Sigma}>\rvert\\
&\leq&\lvert G^{ij}(D_{j}X_{i}^{\Sigma}-\Gamma_{ji}^{k}X_{k}^{\Sigma})\vert+\lvert<G^{ij}(\Gamma_{ji}^{k}X_{k}^{\Sigma}-(\Gamma_{ji}^{k})^{\Sigma}X_{k}^{\Sigma}),\nu_{1}^{\Sigma}>\rvert\\
&=&\lvert G^{ij}(D_{j}X_{i}^{\Sigma}-\Gamma_{ji}^{k}X_{k}^{\Sigma})\vert
\end{eqnarray*}
where $(\Gamma_{ji}^{k})^{\Sigma}$ and $\Gamma_{ji}^{k}$ are Christoffel symbols corresponding to $G_{ij}$ and $g_{ij}$ and $\nu_{1}^{\Sigma}$ is any one of unit normals of $X^{\Sigma}$. The second inequality is because of $<X_{k}^{\Sigma},\nu_{1}^{\Sigma}>=0$.

So the mean curvature vector can be estimated as following

\begin{eqnarray*}
\mathscr{\lvert H}\rvert&\leq&(n+2)\lvert G^{ij}(D_{j}X_{i}^{\Sigma}-\Gamma_{ji}^{k}X_{k}^{\Sigma})\vert\\
&=&(n+2)\lvert G^{ij} (-h_{ij}\nu,h_{ij}^{k}X_{k}-h_{i}^{k}h_{kj}\nu)\rvert\\
&\leq&(n+2)\sqrt{\sum_{i=1}^{n}\frac{\lvert\kappa_{i}\rvert^{2}+\kappa_{i}^{4}}{(1+\kappa_{i}^{2})^{2}}}\leq C.
\end{eqnarray*}

Then we are going to show the almost subharmonic quantity $b$ on the  submanifold $\Sigma$ to satisfy Michael-Simon's mean value inequality.
\begin{theorem}
\label{thm-meanvalue} Suppose $u$ are admissible solutions of equations
(\ref{eq:specialLag}) on $B_{10}\subset\mathbb{R}^{n}$. If $\Theta =\frac{(n-2)\pi}{2}$ or $M$ is convex, then we have for any $y_{0}\in B_{2}$
\begin{equation*}
\sup_{B_{1}}b=b(y_{0})\leq C\int_{B_{1}(y_{0})}b(x) V dx,\label{eq:meanvalue}
\end{equation*}
where $C$ depends only on $\lVert M\rVert_{C^{1}}$ and $n$.
\end{theorem}

\begin{proof}
For convenience, we also choose an orthonormal frame in order to prove this theorem.
First we know from Theorem \ref{lem3}
\begin{eqnarray*}
G^{ij}b_{ij} & \geq & - n .
\end{eqnarray*}
 Let $\chi$ be a non-negative and non-decreasing function in $C^{1}(\mathbb{R})$
with support in the interval $(0,\infty)$. We set 

\[
\psi(r):=\int_{r}^{\infty}t\chi(\rho-t)dt
\]
where $0<\rho<10,$ and $r^{2}:=|X(x)-X(y_{0})|^{2}+2-2<\nu(x),\nu(y_{0})>$. \\
Let us denote 
\[
\mathfrak{B}_{\rho}=\{x\in B_{10}(y_{0}):\quad |X(x)-X(y_{0})|^{2}+2-2<\nu(x),\nu(y_{0})>\leq\rho^{2}\}.
\]
 We may assume that $(X(y_{0}),\nu(y_{0}))=(0,E_{n+1})$. 
First, we have
\begin{equation}
2rr_{i}=2<X,e_{i}>-2\sum_k h_{ik}<e_{k},E_{n+1}>\label{eq:r1}
\end{equation}
and 
\begin{eqnarray}
2r_{i}r_{j}+2rr_{ij} & = & 2\delta_{ij}-2h_{ij}<X,\nu> \nonumber\\
& & -2\sum_k h_{ijk}<e_{k},E_{n+1}>+2\sum_k h_{ik}h_{kj}<\nu,E_{n+1}>.\label{eq:r2} 
\end{eqnarray}
Now we are going to compute the differential inequality of $\psi$,
\begin{eqnarray*}
G^{ij}\psi_{ij} & = & G^{ij}(-r_{i}r\chi(\rho-r))_{j}\\
 & = & -G^{ij}r_{ij}r\chi(\rho-r)-G^{ij}r_{i}r_{j}\chi(\rho-r)+G^{ij}r_{i}r_{j}r\chi^{\prime}(\rho-r).
\end{eqnarray*}
From (\ref{eq:r1}) and (\ref{eq:r2}), we have 
\begin{eqnarray*}
G^{ij}\psi_{ij} & = & -\chi(\rho-r)G^{ij}[\delta_{ij}-h_{ij}<X,\nu>]\\
 &  & +\chi(\rho-r)G^{ij}[\sum_k h_{ijk}<e_{k},E_{n+1}>-\sum_k h_{ik}h_{kj}<\nu,E_{n+1}>]\\
 &  & +G^{ij}r_{i}r_{j}r\chi^{\prime}(\rho-r).
\end{eqnarray*}
Then by \eqref{G1},
we have 
\begin{eqnarray}
G^{ij}\psi_{ij} & = & -n\chi+\chi G^{ij} h_{ij} <X,\nu>\nonumber \\
 &  & +\chi \sum_k G^{ij}h_{ik} h_{kj} (1-<\nu,E_{n+1}>)\nonumber \\
 &  & +G^{ij}r_{i}r_{j}r\chi^{\prime}.\label{eq:sigmaphi}
\end{eqnarray}
By Lemma \ref{Lem9}, we have 
\begin{equation*}
V=\sqrt{\prod^n_{i=1}(1+\kappa_{i}^{2})}.
\end{equation*}
We estimate similarly the second term on the right hand
side of (\ref{eq:sigmaphi}) 
\begin{equation}
\chi G^{ij}h_{ij} <X,\nu>\leq  n \chi r . \label{secondterm}
\end{equation}
It is obvious that
\begin{eqnarray}
\sum_{k=1}^{n}<e_{k},E_{n+1}>^{2} & = & (1-<\nu,E_{n+1}>)(1+<\nu,E_{n+1}>).\label{eq:e_k,E_4^2}
\end{eqnarray}
From the definition of $r$, we see 
\begin{eqnarray}
1-<\nu,E_{n+1}> & \leq & \frac{r^{2}}{2}\label{eq:1-=00005Cnu}
\end{eqnarray}
and
\begin{eqnarray*}
\lvert<e_{k},E_{n+1}>\rvert & \leq & r.\label{eq:e_k,E_4}
\end{eqnarray*} 
From (\ref{eq:1-=00005Cnu}), we deal with the third term of
(\ref{eq:sigmaphi}) 
\begin{equation}
\chi \sum_k G^{ij}h_{ik} h_{kj}(1-<\nu,E_{n+1}>)\leq n \chi r^{2}.\label{thirdterm}
\end{equation}
In sum, we have from \eqref{eq:sigmaphi}, \eqref{secondterm} and \eqref{thirdterm}
\begin{eqnarray}
G^{ij}\psi_{ij} & \leq & -n \chi +n \chi(r^{2}+r) +G^{ij}r_{i}r_{j}r\chi^{\prime}.\label{eq:psi}
\end{eqnarray}

We next claim that 
\begin{equation}
G^{ij}r_{i}r_{j}\leq 1 .\label{eq:cla}
\end{equation}
In fact, due to \eqref{GF}, we know $G^{ij}$ is a linear combination of $[T]^{ij}$. In any orthonormal frame, we have $g_{ij}=\delta_{ij}$. Then, by Lemma \ref{lem7}, we observe that
\begin{equation*}
G^{kl} h_{ki} = G^{ki} h_{kl}.\label{G^klh_ki}
\end{equation*}
Thus we have the following elementary properties 
\begin{eqnarray}
 \delta_{ij}- G^{ij}& = & G^{kl}h_{ki}h_{lj}.\label{eq:fsigma_1-sigma_3}
\end{eqnarray}
By (\ref{eq:fsigma_1-sigma_3}) and \eqref{G^klh_ki}, we estimate as below
\begin{eqnarray*}
    r^2 G^{ij} r_i r_j &=& G^{ij} (<X,e_i>-\sum_k h_{ik}<e_k,E_{n+1}>) (<X,e_j>-\sum_l h_{jl}<e_l,E_{n+1}>) \\
    &=& \sum\limits_{i,j}(\delta_{ij} - G^{kl}h_{ki} h_{lj}) <X,e_i><X,e_j>-2 \sum_l G^{ij} <X,e_i> h_{jl}<e_l,E_{n+1}>\\
    & & + \sum_{i,j} (\delta^{ij}-G^{ij}) <e_i,E_{n+1}><e_j,E_{n+1}>\\
    & =& \sum_i (<X,e_i>^2 +<e_i,E_{n+1}>^2)\\
    & & - G^{ij} (<X,\nu_i>+<e_i,E_{n+1}>)(<X,\nu_j>+<e_j,E_{n+1}>)\\
    &\leq &  \sum_i (<X,e_i>^2 +<e_i,E_{n+1}>^2). 
\end{eqnarray*}
Then by \eqref{eq:e_k,E_4^2} and the definition
of $r$, we have
\begin{eqnarray*}
    G^{ij} r_i r_j &\leq& \frac{\lvert X \rvert^2 - <X,\nu>^2 +(1-<\nu,E_{n+1}>)(1+<\nu,E_{n+1}>) }{r^2}\\
    &\leq & 1.
\end{eqnarray*}
So we have proved the claim (\ref{eq:cla}).\\
We obtain from (\ref{eq:cla}) and (\ref{eq:psi}) that 
\begin{align*}
\triangle_G \psi =G^{ij}\psi_{ij} & \leq -n\chi+n(r^{2}\chi+r\chi)+r\chi^{\prime}.
\end{align*}
Then we multiply both sides by $b$ and take integral on the domain
$\mathfrak{B}_{10}$ 
\begin{eqnarray}
\int_{\mathfrak{B}_{10}}b \triangle_G \psi d\Sigma & \leq & \rho^{n+1}\frac{d}{d\rho}(\int_{\mathfrak{B}_{10}}\frac{b\chi(\rho-r)}{\rho^{n}}d\Sigma)\nonumber \\
 &  & +C(n) \int_{\mathfrak{B}_{10}}rb\chi(\rho-r)d\Sigma .\label{eq:mean}
\end{eqnarray}
By (\ref{eq:logb}), we have 
\begin{equation}
-n \int_{\mathfrak{B}_{10}} \psi d\Sigma \leq \int_{\mathfrak{B}_{10}}b \triangle_G \psi d\Sigma.\label{eq:step3}
\end{equation}
Inserting (\ref{eq:step3}) into (\ref{eq:mean}), we get 
\begin{equation*}
-\frac{d}{d\rho}(\int_{\mathfrak{B}_{10}}\frac{b\chi(\rho-r)}{\rho^{n}} d\Sigma) \leq  \frac{C(n) \int_{\mathfrak{B}_{10}}rb\chi(\rho-r)d\Sigma}{\rho^{n+1}}+ n  \frac{\int_{\mathfrak{B}_{10}} \psi d\Sigma}{\rho^{n+1}}.
\end{equation*}
Because $\chi$, $\chi^{\prime}$ and $\psi$ are all supported in
$\mathfrak{B}_{\rho}$, we deal with right hand side of the above
inequality term by term. For the first term, we have 
\begin{equation}
\frac{\int_{\mathfrak{B}_{10}}rb\chi(\rho-r) d\Sigma}{\rho^{n+1}}\leq\frac{\int_{\mathfrak{B}_{10}}b\chi(\rho-r) d\Sigma}{\rho^{n}}.\label{eq:term1}
\end{equation}
For the second term, we use the definition of $\psi$ to estimate 
\begin{eqnarray}
\frac{\int_{\mathfrak{B}_{10}} \psi d\Sigma}{\rho^{n+1}} & \leq & \frac{\int_{\mathfrak{B}_{10}} b \chi(\rho-r)d\Sigma}{\rho^{n}}.\label{eq:term3}
\end{eqnarray}
We combine (\ref{eq:term1}) and  (\ref{eq:term3})
 with integrating from $\delta$ to $R$:
\begin{eqnarray*}
\int_{\mathfrak{B}_{10}}\frac{b\chi(\delta-r)}{\delta^{n}}d\Sigma & \leq & \int_{\mathfrak{B}_{10}}\frac{b\chi(R-r)}{R^{n}}d\Sigma\\
 &  & +C(n) \int_{\delta}^{R}\frac{\int_{\mathfrak{B}_{10}}b \chi(\rho-r)d\Sigma}{\rho^{n}}d\rho.
\end{eqnarray*}
Then using Gronwall's inequality, we get 
\begin{equation*}
\int_{\mathfrak{B}_{10}}\frac{b \chi(\delta-r)}{\delta^{n}}d\Sigma \leq C(n) \int_{\mathfrak{B}_{10}}\frac{b\chi(R-r)}{R^{n}}d\Sigma.
\end{equation*}
Letting $\chi$ approximate the characteristic function of the interval
$(0,\infty)$, in an appropriate fashion, we obtain, 
\begin{equation*}
\frac{\int_{\mathfrak{B}_{\delta}}bd\Sigma}{\delta^{n}}\leq C(n)\frac{\int_{\mathfrak{B}_{R}}b d\Sigma}{R^{n}}.\label{eq:monotonicity}
\end{equation*}
For a sufficient small $\delta>0$, the geodesic ball with radius
$\delta$ of this submanifold is comparable with $\mathfrak{B}_{\delta}$.
Letting $\delta\rightarrow0$, we finally get 
\begin{eqnarray*}
b(y_{0}) & \leq & C(n)\frac{\int_{\mathfrak{B}_{R}}b d\Sigma}{R^{n}}\leq C(n,\lVert M\rVert_{C^{1}})\frac{\int_{B_{R}(y_{0})}bVdx}{R^{n}}.
\end{eqnarray*}
\end{proof}

\section{Proof of Theorem \ref{thm:SpecialLag}}

In this section, we choose different cutoff functions, all denoted by $0\leq\phi\leq1$, which have support in the larger ball $B_{r+1}(x_{0})$ and equal 1 in the smaller ball $B_{r}(x_{0})$. These functions satisfy $|D\phi|+|D^{2}\phi|\leq C$.
First, we prove that the area is bounded.
\begin{lem}
\label{lem:areaEstimate} If $\kappa\in\Gamma_{n-1}$, then the following
integral is bounded 
\begin{equation*}
    \int_{B_{r}(x_{0})} V dx\leq C,
\end{equation*}

where $C$ depends only on $n$ and $\lVert M\rVert_{C^{1}(B_{r+2}(x_{0}))}$.
\end{lem}

\begin{proof}
For a non-negative cutoff function $\phi\in C_{0}^{\infty}(B_{r+1}(x_{0}))$. $V$ is a linear combination of $\sigma_{k}$ where $k$ from $0$ to $n$.
The integral is obviously bounded when $k=0,1$ as follows 
\begin{eqnarray*}
\int_{B_{r}(x_{0})}\sigma_{1}dx\leq\int_{B_{r+1}(x_{0})}\phi^{2}\sigma_{1}dx & = & \int_{B_{r+1}}\phi^{2}div(\frac{Du}{W})dx\\
 & = & \int_{B_{r+1}}-\sum_{i}(\phi^{2})_{i}\frac{u_{i}}{W}dx\\
 & \leq & C.
\end{eqnarray*}
Inductively, we assume that 
\[
\int_{B_{r+1}}\sigma_{k-1}dx\leq C.
\]
Let us prove that for any $k\leq n-1$ 
\begin{equation}
\int_{B_{r}}\sigma_{k}dx\leq C.\label{eq:sigmak}
\end{equation}
Similarly, for $\kappa\in\Gamma_{n-1}$ 
\begin{eqnarray*}
\int_{B_{r}}k\sigma_{k}dx\leq\int_{B_{r+1}}k\phi^{2}\sigma_{k}dx & = & \int_{B_{r+1}}\phi^{2}[T_{k-1}]_{i}^{j}D_{j}(\frac{u_{i}}{W})dx\\
 & \leq & -\int_{B_{r+1}}(\phi^{2})_{j}[T_{k-1}]_{i}^{j}\frac{u_{i}}{W}dx\\
 & \leq & C\int_{B_{r+1}}\sigma_{k-1}dx.
\end{eqnarray*}
By induction assumption, we get the estimate for $k\leq n-1$.

Then we estimate the term with $\sigma_{n}$. 
\begin{eqnarray*}
-n\int_{B_{r+1}(x_{0})}\phi^{2}\sigma_{n}dx & = & -\int\phi^{2}[T_{n-1}]_{i}^{j}D_{j}(\frac{u_{i}}{W})dx\\
 & = & 2\int\phi[T_{n-1}]_{i}^{j}\phi_{j}\frac{u_{i}}{W}dx.
\end{eqnarray*}
Using (\ref{eq:Tk1}), we continue our estimate 
\begin{eqnarray*}
& &\int \phi[T_{n-1}]_{i}^{j}\phi_{j}\frac{u_{i}}{W}dx\\
& = & \int\phi\phi_{i}\frac{u_{i}}{W}\sigma_{n-1}dx-\int\phi[T_{n-2}]_{i}^{k}\phi_{j}\frac{u_{i}}{W}D_{k}(\frac{u_{j}}{W})dx.\\
 & = & \int\phi\phi_{i}\frac{u_{i}}{W}\sigma_{n-1}dx+\int[T_{n-2}]_{i}^{k}(\phi\phi_{j})_{k}\frac{u_{i}}{W}\frac{u_{j}}{W}dx+
 (n-1)\int\sigma_{n-1}\phi\phi_{j}\frac{u_{j}}{W}dx\\
 & \leq & C \int \phi \sigma_{n-1} \leq  C
\end{eqnarray*}
where we use the previous estimates (\ref{eq:sigmak}) in the last
inequality. So we get 
\[
\int_{B_{r}(x_{0})} V dx\leq\int_{B_{r+1}(x_{0})}\phi^{2}[\cos \Theta \sum\limits_{0 \leq 2k \leq n} (-1)^k \sigma_{2k}  +\sin \Theta \sum\limits_{1\leq 2k+1 \leq n} (-1)^k  \sigma_{2k+1}]dx\leq C.
\]
\end{proof}

\begin{lem} \label{lem19} If $\kappa\in\Gamma_{n-1}$ and satisfies the equation \eqref{eq:specialLag}, then for $\forall \quad k\leq n-1$, we have
\begin{equation*}
 [T_{k-1}]^{iq} g_{pq} [T_{k-1}]^{jp} \leq C(n) G^{ij} V^2 .
\end{equation*}
\end{lem}
\begin{proof}
  We know that $ [T_{k-1}]^{ij} $ and $G^{ij}$ are positive definite matrix when $k\leq n-1$. Without loss of generality, we can assume pointwise that $g_{ij}$ and $h_{ij}$ are diagonalized. Then $[T_{k-1}]^{ii}$ consists at most $(n-2)$
eigenvalues without $\kappa_{i}$. Each component of $[T_{k-1}]^{ii}$
does not contain $\kappa_{i}$ and some other eigenvalues. So at this point we obtain from Lemma \ref{Lem9} that
\begin{eqnarray*}
 [T_{k-1}]^{ii} & \leq &\sum\limits_{\overset{i_1< i_2 \cdots < i_{k-1} }{i_1,i_2,\cdots,i_{k-1} \neq i}} \lvert \kappa_{i_1} \rvert \lvert \kappa_{i_2} \rvert\cdots \lvert \kappa_{i_{k-1}} \rvert \\
&\leq& C(n) \sum\limits_{j\neq i} \frac{V}{\sqrt{(1+\kappa^2_i)(1+\kappa^2_j)}}\\
&\leq & C(n) \sqrt{G^{ii}} V.
\end{eqnarray*}
Thus we have 
\begin{equation*}
 [T_{k-1}]^{iq} g_{pq} [T_{k-1}]^{jp} \leq C(n) G^{ij} V^2 .
\end{equation*}
\end{proof}

Then we start to prove Theorem \ref{thm:SpecialLag}.
\begin{proof}
From Theorem \ref{thm-meanvalue}, we have at the maximum point $x_{0}$
of $\bar{B}_{1}(0)$ 
\begin{eqnarray}
b(x_{0}) & \leq & C \int_{B_{1}(x_{0})}b  V dx.\label{eq:WY1}
\end{eqnarray}

We shall estimate $\int_{B_{1}(x_{0})}b\sigma_{1}dx$ in the above
integral at first. Recall that 
\begin{equation}
G^{ij}b{}_{ij}\geq c_{0}G^{ij}b{}_{i}b{}_{j}-n,\label{eq:logb-1}
\end{equation}
we have an integral version of this inequality for any $r<5$, 
\begin{eqnarray}
\int_{B_{r+1}}-G^{ij}\phi_{i}b_{j}d\Sigma & \geq & c_{0}\int_{B_{r+1}}\phi G^{ij}b_{i}b_{j}d\Sigma-n\int_{B_{r+1}} d\Sigma,\label{eq:intlog}
\end{eqnarray}
for all non-negative $\phi\in C_{0}^{\infty}$. 
\begin{eqnarray}
\int_{B_{1}(x_{0})}b\sigma_{1}dx & \leq & \int_{B_{2}(x_{0})}\phi b\sigma_{1}dx\nonumber \\
 & \leq & C(\int_{B_{2}(x_{0})}bdx+\int_{B_{2}(x_{0})}|Db|dx)\nonumber \\
 & \leq & C(1+\int_{B_{2}(x_{0})}|Db|dx).\label{eq:WY2}
\end{eqnarray}
 We only need to estimate $\int_{B_{2}(x_{0})}|Db|dx$. Using Lemma \ref{lem19}, we have that
\begin{equation*}
    F^{ij}V \geq g^{ij} \geq \frac{\delta_{ij}}{W^2}.
\end{equation*}
Then we  obtain 
\begin{equation*}
    \int_{B_{2}(x_{0})}|Db|dx\leq \int_{B_{2}(x_{0})} W \sqrt{F^{ij}b_{i}b_{j}V}dx.
\end{equation*}
By Holder inequality, we have 
\begin{eqnarray}
\int_{B_{2}(x_{0})}|Db|dx & \leq & C (\int_{B_{2}(x_{0})}G^{ij}b_{i}b_{j}d\Sigma)^{\frac{1}{2}}(\int_{B_{2}(x_{0})}Vdx)^{\frac{1}{2}}\nonumber \\
 & \leq & C\int_{B_{3}(x_{0})}\phi^{2}G^{ij}b_{i}b_{j}d\Sigma+\int_{B_{2}}  V dx.\label{eq:WY3}
\end{eqnarray}
 Then using (\ref{eq:intlog}) and Lemma \ref{lem:areaEstimate},
we get 
\begin{eqnarray*}
\int_{B_{3}(x_{0})}\phi^{2}G^{ij}b_{i}b_{j}d\Sigma & \leq & C(-\int_{B_{3}(x_{0})}\phi G^{ij}\phi_{i}b_{j}d\Sigma+1)\\
 & \leq & C(\int_{B_{3}(x_{0})}\sqrt{\phi^{2}G^{ij}b_{i}b_{j}}\sqrt{G^{kl}\phi_{k}\phi_{l}}d\Sigma+1).
\end{eqnarray*}
 By Cauchy-Schwarz inequality 
\begin{eqnarray*}
\int_{B_{3}(x_{0})}\phi^{2}G^{ij}b_{i}b_{j}d\Sigma & \leq & C(\epsilon\int_{B_{3}(x_{0})}\phi^{2}G^{ij}b_{i}b_{j}d\Sigma+\int_{B_{3}(x_{0})}G^{ij}\phi_{i}\phi_{j}d\Sigma+1)\\
 & \leq & C\epsilon\int_{B_{3}(x_{0})}\phi^{2}G^{ij}b_{i}b_{j}d\Sigma+C\int_{B_{3}(x_{0})}V dx+C.
\end{eqnarray*}
We choose $\epsilon$ small such that $C\epsilon\leq\frac{1}{2}$
and apply Lemma \ref{lem:areaEstimate},
\begin{eqnarray}
\int_{B_{3}(x_{0})}\phi^{2}G^{ij}b_{i}b_{j}d\Sigma & \leq & C.\label{eq:WY4}
\end{eqnarray}
By combining (\ref{eq:WY2}), (\ref{eq:WY3}), (\ref{eq:WY4}) and  Lemma \ref{lem:areaEstimate},
we have 
\begin{equation}
\int_{B_{1}(x_{0})}b\sigma_{1}dx\leq C_1.\label{eq:part1}
\end{equation}

The second part is to estimate $\int_{B_{1}(x_{0})}b\sigma_{k}dx$
for any $k\leq n-1$ inductively. Suppose we have already the estimate
\begin{equation}
\int_{B_{2}(x_{0})}b\sigma_{k-1}dx\leq C_{k-1}.\label{eq:k-1}
\end{equation}
We are going to prove that 
\[
\begin{array}{cccc}
\int_{B_{1}(x_{0})}b\sigma_{k}dx\leq C_{n-1}, & for & any & k\leq n-1\end{array}.
\]
 Thanks to the divergence free property, we integral by parts as follows
\begin{eqnarray}
k\int_{B_{2}(x_{0})}\phi^{2}b\sigma_{k}dx & = & \int\phi^{2}b[T_{k-1}]_{i}^{j}D_{j}(\frac{u_{i}}{W})dx\nonumber \\
 & = & -\int[T_{k-1}]_{i}^{j}(\phi^{2})_{j}b\frac{u_{i}}{W}dx-\int[T_{k-1}]_{i}^{j}\phi^{2}b_{j}\frac{u_{i}}{W}dx.\label{eq:sigma3}
\end{eqnarray}
By induction assumption (\ref{eq:k-1}), it is easy to see that 
\begin{equation}
    -\int[T_{k-1}]_{i}^{j}(\phi^{2})_{j}b\frac{u_{i}}{W}dx\leq C\int_{B_{2}}\sigma_{k-1}bdx \leq C. \label{eqindunctionk-1}
\end{equation}
From Cauchy-Schwarz inequality and Lemma \ref{lem19}, we have 
\begin{eqnarray}
|[T_{k-1}]_{i}^{j}b_{j}\frac{u_{i}}{W}| &=&|[T_{k-1}]^{jp} g_{pi} b_{j}\frac{u_{i}}{W}| \nonumber\\
& \leq & \frac{1}{2} \frac{b_l [T_{k-1}]^{lq} g_{pq} [T_{k-1}]^{jp}b_j}{V} +\frac{1}{2} V  \frac{u_i}{W} g_{ij} \frac{u_j}{W} \nonumber\\
&\leq & C(n, || M ||_{C^1} ) (G^{ij}b_i b_j V + V).\label{WYn-1}
\end{eqnarray}
From \eqref{WYn-1}, (\ref{eq:WY4}) and Lemma \ref{lem:areaEstimate}, we have
\begin{equation}
    -\int [T_{k-1}]^j_i b_j \phi^2 \frac{u_i}{W} dx \leq C \int G^{ij} b_i b_j d \Sigma + \int V dx \leq C \quad \forall \quad k\leq n-1. \label{T_n-1}
\end{equation}
Thus from \eqref{eq:sigma3}, \eqref{eqindunctionk-1} and  \eqref{T_n-1}, we get the estimate for any $k\leq n-1$ 
\begin{equation}
    \int_{B_{1}(x_{0})}b\sigma_{k}dx\leq C_{n-1}.\label{n-1}
\end{equation}

The last part is the estimate for $\int_{B_{1}(x_{0})}-b\sigma_{n}dx$. We divided the analysis into two cases to handle this term separately.\\  Recall that $\theta = \Theta - \frac{(n-1)\pi}{2}$.

Case 1:  $\lvert \cos \theta  \rvert \geq \frac{\sqrt{3}}{2}$.\\
We utilize the equation \eqref{eq:SpecialLag3} to reduce the estimation of the term with $\sigma_n$ to terms involving $\sigma_k$ for $k<n$. 
From the equation \eqref{eq:SpecialLag3} in Lemma \ref{lem_theta}, we observe that
\begin{equation*}
    \sigma_n = \frac{\sin \theta}{\cos \theta} \sum\limits_{1\leq 2k+1\leq n} (-1)^k \sigma_{n-2k-1} - \sum\limits_{2\leq 2k \leq n} (-1)^k \sigma_{n-2k}.
 \end{equation*}
So 
\begin{equation*}
    \int_{B_{1}(x_{0})}-b\sigma_{n}dx \leq  C \sum\limits^{n-1}_{k=1}\int_{B_{1}(x_{0})}b\sigma_{k}dx\leq C(n, C_{n-1}).
\end{equation*}

Case 2: $ \lvert \sin \theta \rvert \geq \frac{1}{2}$.\\
First, we integrate by parts once to obtain 
\begin{eqnarray*}
n\int_{B_{2}(x_{0})}-\phi^{2}b\sigma_{n}dx & = & -\int\phi^{2}b[T_{n-1}]_{i}^{j}D_{j}(\frac{u_{i}}{W})dx\\
 & \leq & \underbrace{\int[T_{n-1}]_{i}^{j}(\phi^{2})_{j}b\frac{u_{i}}{W}}_{I}dx+\underbrace{\int[T_{n-1}]_{i}^{j}\phi^{2}b_{j}\frac{u_{i}}{W}}_{II}dx.
\end{eqnarray*}
We estimate $I$ by applying (\ref{eq:Tk1}) and integrating by parts, 
\begin{eqnarray}
I & = & \int(\sigma_{n-1}\delta_{i}^{j}-[T_{n-2}]_{i}^{k}h_{k}^{j})(\phi^{2})_{j}b\frac{u_{i}}{W}dx\nonumber \\
 & \leq & \int b\sigma_{n-1}dx-\int[T_{n-2}]_{i}^{k}D_{k}(\frac{u_{j}}{W})(\phi^{2})_{j}b\frac{u_{i}}{W}dx\nonumber \\
 & \leq & C+\int[T_{n-2}]_{i}^{k}\frac{u_{j}}{W}(\phi^{2})_{jk}b\frac{u_{i}}{W}dx+\int[T_{n-2}]_{i}^{k}\frac{u_{j}}{W}(\phi^{2})_{j}b_{k}\frac{u_{i}}{W}dx\nonumber \\
 &  & +(n-1)\int\sigma_{n-1}\frac{u_{j}}{W}(\phi^{2})_{j}bdx\nonumber \\
 & \leq & C+\int\sigma_{n-2}bdx+\int[T_{n-2}]^{kl}b_{k}g_{li}\frac{u_{i}}{W}\frac{u_{j}}{W}(\phi^{2})_{j}dx.\label{eq:I}
\end{eqnarray}
The last terms of (\ref{eq:I}) can be estimated by the same argument as \eqref{T_n-1}
before. \\
So we obtain 
\begin{equation}
I=\int[T_{n-1}]_{i}^{j}(\phi^{2})_{j}b\frac{u_{i}}{W}dx\leq C.\label{eq:Ifinal}
\end{equation}
By \eqref{eq:Tk2}, we have
\begin{eqnarray*}
II & = & \int[T_{n-1}]_{i}^{j}\phi^{2}b_{j}\frac{u_{i}}{W}dx\nonumber \\
& =&   \int \sum_i \sigma_{n-1} \phi^2 b_i \frac{u_i}{W} dx -\int \sum_i [T_{n-2}]^j_l h^l_i  \phi^2 b_j \frac{u_i}{W} dx.
\end{eqnarray*}
Then from the equation \eqref{eq:SpecialLag3} in Lemma \ref{lem_theta}, we have that
\begin{equation}
    \sigma_{n-1} = -  \sum\limits_{3\leq 2k+1\leq n} (-1)^k \sigma_{n-2k-1} +\frac{\cos \theta}{\sin \theta} \sum\limits_{0\leq 2k \leq n} (-1)^k \sigma_{n-2k}.\label{sigma_n-1}
\end{equation}
Moreover (recall [$T_{-1}]=0$), we have
\begin{equation*}
    F^{j}_i = \frac{\partial F}{\partial h^i_j} = \cos \theta \sum\limits_{0\leq 2k \leq n} (-1)^k [T_{n-2k-1}]^j_i -\sin \theta \sum\limits_{1\leq 2k+1 \leq n} (-1)^k [T_{n-2k-2}]^j_i.
\end{equation*}
Thus we have 
\begin{equation}
    [T_{n-2}]^j_i  = -\frac{1}{\sin \theta } F^j_i +\frac{\cos \theta}{\sin \theta} \sum\limits_{0\leq 2k \leq n} (-1)^k [T_{n-2k-1}]^j_i - \sum\limits_{3\leq 2k+1 \leq n} (-1)^k [T_{n-2k-2}]^j_i.\label{Tn-2}
\end{equation}
We observe from \eqref{sigma_n-1}, \eqref{Tn-2} and \eqref{eq:Tk2} that the terms involving $T_{n-1}$ and $\sigma_n$ cancel each other out, as follows:
\begin{eqnarray}
      \sigma_{n-1} \delta^j_i - [T_{n-2}]^j_l h^l_i &=& -  \sum\limits_{3\leq 2k+1\leq n} (-1)^k \big \{ \sigma_{n-2k-1} \delta^j_i-[T_{n-2k-2}]^j_l h^l_i  \big\} \nonumber\\
      & &+ \frac{\cos \theta}{\sin \theta } \sum\limits_{0\leq 2k \leq n} (-1)^k \big \{ \sigma_{n-2k} \delta^j_i -  [T_{n-2k-1}]^j_l h^l_i \big\} \nonumber\\
      & &+\frac{1}{\sin \theta } F^j_l h^l_i  \nonumber \\
      &=& -  \sum\limits_{3\leq 2k+1\leq n} (-1)^k [T_{n-2k-1}]^j_i  + \frac{\cos \theta}{\sin \theta } \sum\limits_{2\leq 2k \leq n} (-1)^k [T_{n-2k}]^j_i \nonumber\\
      & & +\frac{1}{\sin \theta } F^j_l h^l_i  \label{sigma_n+T_n-1}.
\end{eqnarray}
 Now using \eqref{sigma_n+T_n-1}, we handle $II$ as follows: 
\begin{eqnarray*}
   II & =& \int \phi^2 \sum_i \big\{ -  \sum\limits_{3\leq 2k+1\leq n} (-1)^k [T_{n-2k-1}]^j_i  + \frac{\cos \theta}{\sin \theta } \sum\limits_{2\leq 2k \leq n} (-1)^k [T_{n-2k}]^j_i \big \} b_j \frac{u_i}{W}  dx\\
& &+ \int \frac{1}{\sin \theta} \sum_i F^j_l h^l_i  \phi^2 b_j \frac{u_i}{W} dx.
\end{eqnarray*}
Because of the previous argument in \eqref{T_n-1}, we have
\begin{equation*}
   \int \phi^2 \sum_i \big\{ -  \sum\limits_{3\leq 2k+1\leq n} (-1)^k [T_{n-2k-1}]^j_i  + \frac{\cos \theta}{\sin \theta } \sum\limits_{2\leq 2k \leq n} (-1)^k [T_{n-2k}]^j_i \big \} b_j \frac{u_i}{W}  dx \leq C.
\end{equation*}
Thus we have 
\begin{equation}
    II \leq C + \frac{1}{\sin \theta} \int \sum_i F^j_l  h^l_i  \phi^2 b_j \frac{u_i}{W} dx\leq C + 2\lvert \int F^{jk}h_{ki}\phi^{2}b_{j}\frac{u_{i}}{W}dx \rvert. \label{eq:II}
\end{equation}
We compute the second term of \eqref{eq:II}  
\begin{eqnarray*}
\lvert \int F^{jk}h_{ki}\phi^{2}b_{j}\frac{u_{i}}{W}dx  \rvert & \leq & 2\int\phi^{2}F{}^{ji}b_{j}b_{i}dx+2\int F{}^{ij}h_{ik}\frac{u_{k}}{W}h_{jl}\frac{u_{l}}{W}\phi^{2}dx\\
 & \leq & 2\int\phi^{2}G^{ij}b_{j}b_{i}d\Sigma+2\int Vg_{lk} \frac{u_k}{W} \frac{u_l}{W} \phi^2 dx-2\int F^{l}_k \frac{u_{k}}{W} \frac{u_{l}}{W}\phi^{2}dx.
\end{eqnarray*}
By (\ref{eq:WY4}) and Lemma \ref{lem:areaEstimate}, we get the estimate
for $II$, 
\begin{equation}
II=\int[T_{n-1}]_{i}^{j}\phi^{2}b_{j}\frac{u_{i}}{W}dx\leq C.\label{eq:IIfinal}
\end{equation}
With the estimate (\ref{eq:Ifinal}) and (\ref{eq:IIfinal}) for $I$
and $II$ , we get 
\begin{equation}
\int_{B_{1}(x_{0})}-b\sigma_{n}dx\leq C.\label{eq:part2}
\end{equation}

Finally, combining (\ref{eq:part1}), \eqref{n-1} and (\ref{eq:part2}), we get
the estimate 
\begin{equation*}
 b(x_{0})\leq C.
\end{equation*}
\end{proof}

\section{Discussion in dimension two} \label{OTinD2}

The equation in dimension two is 
\begin{equation*}
    \cos \Theta H - \sin \Theta  (1- K)=0. \label{SpecialLM2}
\end{equation*}
It is well known that in graph case, the mean curvature $H$ and the Gaussian curvature $K$ are:
\begin{eqnarray*}
     H &=& \frac{1}{W} g^{ij} u_{ij}\\
     K &=& \frac{\det D^2 u}{W^4}. 
\end{eqnarray*}
Thus the special Lagrangian curvature equation can be written as 
\begin{equation}
    \det [u_{ij} + W \cot \Theta (\delta_{ij}+u_i u_j)] =\frac{W^4}{\sin^2 \Theta}, \label{specialLC2}
\end{equation}
with $D^2 u + W \cot \Theta g \geq 0$.\\
\begin{lem} \label{optimalTransPro}
 The equation \eqref{specialLC2} can be derived from an optimal transportation problem with cost function $c(x,y)= -\sqrt{\tan^2 \Theta -|x-y|^2}$ and densities $f=\frac{1}{\cos^2 \Theta}$ and $g=1$.
\end{lem}
\begin{proof}
    Let $\Omega$, $\Omega^\ast$ be two bounded domain in $\mathbb{R}^n$, and let $f$,$g$ be two nonnegative functions defined on $\Omega$ and $\Omega^\ast$, and satisfying
\begin{equation*}
\int_\Omega f(x) dx= \int_{\Omega^\ast} g(y) dy.
\end{equation*}
Monge's optimal transportation problem concerns the existence of a measure preserving mapping $T: \Omega \rightarrow \Omega^\ast$ that minimizes
\begin{equation*}
  \mathscr {M} = \inf\limits_T \{ \int_\Omega c(x,T(x)) f(x) dx : \int_{T^{-1}(E)} f dx =  \int_{E} g dy \quad \forall E\subset \Omega^\ast \}.
\end{equation*}
The dual Monge-Kantorovitch problem is to find an optimal pair of potentials $(u,v)$ that realizes
\begin{equation*}
    \mathscr {K} =\sup\limits_{(u,v)} \big\{ -\int_{\Omega} u(x) f(x)dx - \int_{\Omega^\ast} v(y) g(y) dy: u(x)+v(y)\geq -c(x,y) , \forall x \in \Omega, y\in \Omega^\ast \big\}.
\end{equation*}
A fundamental relation is
\begin{equation*}
    \mathscr {M}=\mathscr {K}.
\end{equation*}
Let $(u,v)$ be the maximizer attains $ \mathscr {K}$. We may assume the maximizer $(u,v)$ satisfies the relation:
\begin{eqnarray*}
    u(x) &=& \sup\limits_{y\in \Omega^\ast} \{ -c(x,y)-v(y) \},\\
    v(y) &=& \sup\limits_{x\in \Omega} \{ -c(x,y) - u(x) \}.
\end{eqnarray*}
A function $u$ is c-convex if there exists another function $v$ such that 
\begin{equation*}
     u(x) = \sup\limits_{y\in \Omega^\ast} \{ -c(x,y)-v(y) \}.
\end{equation*}
For any given $x_0 \in \Omega$, there exists $y_0 \in \overline{\Omega^\ast}$ such that
\begin{eqnarray*}
    u(x_0) &= &-c(x_0,y_0) - v(y_0),\\
    u(x) &\geq& -c(x,y_0)-v(y_0) \quad \forall x\in\Omega.
\end{eqnarray*}
The above two inequalities tell us that $x_0$ is a global minimum point of function
\begin{equation*}
   h(x):= u(x)+c(x,y_0)+v(y_0).
\end{equation*}
Thus if $Dh(x_0)$ and $D^2 h(x_0)$ exist, we have 
\begin{eqnarray*}
    Dh(x_0) &=& 0,\\
    D^2 h(x_0) & \geq& 0.
\end{eqnarray*}
For the cost function $c(x,y)= -\sqrt{\tan^2 \Theta-|x-y|^2}$, we have
\begin{eqnarray}
    u_i(x_0) &=& -D_{x_i} c(x_0,y_0)= - \frac{x_0^i - y_0^i}{\sqrt{\tan^2 \Theta-|x_0-y_0|^2}}, \label{criticalOT}\\
    u_{ij}(x_0) &\geq& -D^2_{x_i x_j} c(x_0,y_0) \nonumber \\
    & &= -\frac{\delta_{ij}}{\sqrt{\tan^2 \Theta-|x_0-y_0|^2}} -\frac{(x^i_0 - y^i_0)(x^j_0 - y^j_0)}{(\tan^2 \Theta-|x_0-y_0|^2)^\frac{3}{2}}.\label{minOT}
\end{eqnarray}
The optimal map is denoted by $T_u$ such that $T_u(x_0) = y_0$. From the proof in \cite{GangboMcCann}, the map is measure preserving such that
\begin{equation}
    \det DT_u(x) = \frac{f(x)}{g(T_u(x))}.\label{MeasurePreserving}
\end{equation}
From \eqref{criticalOT} and  \eqref{minOT}, we have for almost all $x\in \Omega$
\begin{eqnarray}
    \sqrt{\tan^2 \Theta-|x-T_u(x)|^2} &= &\frac{\tan \Theta}{W(x)}, \quad for \quad 0<\Theta < \frac{\pi}{2} \nonumber\\
    T_u(x)&=& \tan \Theta \frac{Du(x)}{W(x)}  +x, \nonumber\\
    D^2 u &\geq&  -  \cot{\Theta} W g. \label{c-convex}
\end{eqnarray}
If $u\in C^2(\Omega)$, c-convexity is equivalent to \eqref{c-convex}.
From \eqref{criticalOT}, we have 
\begin{equation*}
    D^2 u(x) = -D^2_{xx} c (x, T_u (x)) - D^2_{xy} c \cdot DT.
\end{equation*}
Hence the equation satisfied by $u$ is 
\begin{equation}
    \det (D^2_{xx} c + D^2 u) = \det (-D^2_{xy} c) \det DT. \label{OTequation}
\end{equation}
By direct computation, we have
\begin{eqnarray}
    D^2_{x_i x_j} c (x,T_u(x))&=& \frac{\delta_{ij}}{\sqrt{\tan^2 \Theta-|x-T_u|^2}} + \frac{(x^i-T^i_u(x))(x^j-T^j_u(x))}{(\tan^2 \Theta-|x-T_u|^2)^\frac{3}{2}} \nonumber\\
    &=& \cot{\Theta} W (\delta_{ij} + u_i u_j), \label{C_xixj}
\end{eqnarray}
and 
\begin{eqnarray}
    D^2_{x_i y_j} c (x,T_u(x))&=& -\frac{\delta_{ij}}{\sqrt{\tan^2 \Theta-|x-T_u|^2}} - \frac{(x^i-T^i_u(x))(x^j-T^j_u(x))}{(\tan^2 \Theta-|x-T_u|^2)^\frac{3}{2}} \nonumber \\
    &=&- \cot{\Theta} W (\delta_{ij} + u_i u_j).\label{C_xiyj}
\end{eqnarray}
Plugging \eqref{C_xixj}, \eqref{C_xiyj}, and \eqref{MeasurePreserving} into equation \eqref{OTequation}, we obtain
\begin{equation}
    \det (D^2 u + \cot{\Theta} W g) = \cot^2{\Theta} W^4 \frac{f(x)}{g(T_u(x))},\label{OTE}
\end{equation}
with the c-convex condition $D^2 u +\cot{\Theta} W g \geq 0$.
When $\frac{\pi}{2}>\Theta >0 $, our equation \eqref{specialLC2} has the same form as \eqref{OTE} with density  $f=\frac{1}{\cos^2 \Theta}$ and $g=1$.
\end{proof}

We are currently verifying that equation \eqref{OTE} does not satisfy the well-known Ma-Trudinger-Wang condition from \cite{ma2005regularity}.
Denote 
\begin{equation*}
    A_{ij} (x,Du) := -D^2_{xx} c (x,T_u(x)) = - \cot{\Theta} W (\delta_{ij}+ u_i u_j).
\end{equation*}
We have
\begin{eqnarray*}
    D_{u_k} A_{ij}(x,Du) = - \cot{\Theta} W (\delta_{ik} u_j + u_i \delta_{jk}) - \cot{\Theta} \frac{u_k}{W}(\delta_{ij}+u_i u_j)
\end{eqnarray*}
and 
\begin{eqnarray*}
    D^2_{u_k u_l} A_{ij}(x,Du) &=& -\cot{\Theta} \frac{u_l}{W} (\delta_{ik} u_j + u_i \delta_{jk}) - \cot{\Theta} W(\delta_{ik} \delta_{jl} + \delta_{il} \delta_{jk}) -\cot{\Theta} \frac{\delta_{kl}}{W} (\delta_{ij}+u_i u_j) \\
   & &+ \cot{\Theta} \frac{u_k u_l}{W^3} (\delta_{ij} + u_i u_j) - \cot{\Theta} \frac{u_k}{W} (\delta_{il} u_j + u_i \delta_{jl}).
\end{eqnarray*}
Thus 
\begin{eqnarray*}
    \mathscr{A}(x,y) (\xi,\nu):= D^2_{u_k u_l} A_{ij}(x,Du) \xi_i \xi_j \nu_k \nu_l (x,Du) = - \cot{\Theta} \frac{g_{ij} \xi_i \xi_j g^{kl} \nu_k \nu_l }{W} <0, \quad \forall \xi \perp \nu. 
\end{eqnarray*}
Therefore, equation \eqref{specialLC2} violates the Ma-Trudinger-Wang condition.

\section{Gradient estimates } \label{gradientestimateSec}
In this section we prove a gradient estimate that will complete our story. Notice that our theorem holds for all constant phases.  For related results we refer to Sheng-Trudinger-Wang \cite{ShengUrbasWang04} and Warren-Yuan \cite{WY10}.
\begin{theorem}
\label{thm:Gradestimate2} Suppose $M$ is a smooth graph over $B_{1}\subset\mathbb{R}^{n}$
and it is a solution of equation (\ref{eq:specialLag}). Then we have
\begin{equation*}
|Du(0)|\leq C(n)\mathop{osc}\limits_{B_1(0)}u.
\end{equation*}
\end{theorem}
\begin{proof}
By scaling we assume that $1\le u\le 2$. We know that $X=(x,u)$, so $u=<X,E_{n+1}>$ and
$\lvert X \rvert^2 = \lvert x\rvert^2 + u^2 $.
The cut off function $\eta$ can be written to
\begin{equation*}
    \eta = 1- \lvert X \rvert ^2 +<X,E_{n+1}>^2.
\end{equation*} 
Denote $W:=\sqrt{1+\lvert D u\rvert^2} $.
We also know that the outer normal vector of the graph is $\nu = (\frac{u_i}{W} , -\frac{1}{W})$, so 
\begin{equation*}
    W= -\frac{1}{<\nu,E_{n+1}>}.
\end{equation*}
We consider 
\begin{equation*}
  P=2\log \eta + \log \log \frac{-1}{<\nu,E_{n+1}>} + <X,E_{n+1}> .
\end{equation*}
Suppose $P$ attains maximum at one point $x_0 \in B_1(0)$, we have
\begin{equation*}
    0=P_i=\frac{2\eta_i}{\eta} - \frac{\sum_k h_{ik}<e_k,E_{n+1}>}{<\nu,E_{n+1}> \log W} +  <e_i,E_{n+1}>
\end{equation*}
and
\begin{eqnarray*}
     0\geq P_{ij}&= &2\frac{\eta_{ij}}{\eta} -2\frac{\eta_i \eta_j}{\eta^2}- \frac{\sum_k h_{ikj}<e_k,E_{n+1}>}{<\nu,E_{n+1}> \log W} \\
    &  &+\frac{\sum_k h_{ik} h_{kj} <\nu,E_{n+1}>}{<\nu,E_{n+1}> \log W}+\frac{(\sum_k h_{ik}<e_k,E_{n+1}>)(\sum_l h_{jl}<e_l,E_{n+1}>)}{<\nu,E_{n+1}>^2 \log W}\\
    & &-\frac{(\sum_k h_{ik}<e_k,E_{n+1}>)(\sum_l h_{jl}<e_l,E_{n+1}>)}{<\nu,E_{n+1}>^2\log^2 W}- h_{ij} <\nu, E_{n+1}>.
\end{eqnarray*}
Contracting  with $G^{ij}$, we have
\begin{eqnarray*}
     0\geq G^{ij} P_{ij}&=&2\frac{G^{ij}\eta_{ij}}{\eta} -2\frac{G^{ij}\eta_i \eta_j}{\eta^2}- \frac{\sum_k G^{ij} h_{ikj}<e_k,E_{n+1}>}{<\nu,E_{n+1}> \log W} \\
    &  &+\frac{\sum_k G^{ij} h_{ik} h_{kj} }{ \log W}+\frac{ G^{ij}(\sum_k h_{ik}<e_k,E_{n+1}>)(\sum_l h_{jl}<e_l,E_{n+1}>)}{<\nu,E_{n+1}>^2 \log W}\\
    & &-\frac{ G^{ij}(\sum_k h_{ik}<e_k,E_{n+1}>)(\sum_l h_{jl}<e_l,E_{n+1}>)}{<\nu,E_{n+1}>^2\log^2 W}\\
    & &-  G^{ij} h_{ij} <\nu, E_{n+1}>.
\end{eqnarray*}
Suppose $\lvert Du(x_0) \rvert$ is large, we have 
\begin{eqnarray*}
    & &\frac{ G^{ij}(\sum_k h_{ik}<e_k,E_{n+1}>)(\sum_l h_{jl}<e_l,E_{n+1}>)}{<\nu,E_{n+1}>^2 \log W}\\
    & &-\frac{ G^{ij}(\sum_k h_{ik}<e_k,E_{n+1}>)(\sum_l h_{jl}<e_l,E_{n+1}>)}{<\nu,E_{n+1}>^2\log^2 W}\\
    &\geq & \frac{1}{2} \frac{ G^{ij}(\sum_k h_{ik}<e_k,E_{n+1}>)(\sum_l h_{jl}<e_l,E_{n+1}>)}{<\nu,E_{n+1}>^2 \log W}.
\end{eqnarray*}
By the equation (\ref{eq:specialLag}) we have 
\begin{eqnarray*}
    - \frac{\sum_k G^{ij} h_{ikj}<e_k,E_{n+1}>}{<\nu,E_{n+1}> \log W} +\frac{\sum_k G^{ij} h_{ik} h_{kj} <\nu,E_{n+1}>}{<\nu,E_{n+1}> \log W}\geq0.
\end{eqnarray*}
Thus we have
\begin{eqnarray*}
    G^{ij}P_{ij}& \geq &2\frac{G^{ij}\eta_{ij}}{\eta} -2\frac{G^{ij}\eta_i \eta_j}{\eta^2}+\frac{1}{2}\frac{ G^{ij}(\sum_k h_{ik}<e_k,E_{n+1}>)(\sum_l h_{jl}<e_l,E_{n+1}>)}{<\nu,E_{n+1}>^2 \log W}\\
    & &  - G^{ij} h_{ij} <\nu, E_{n+1}>.
\end{eqnarray*}
Choose $\{ e_i \}_{i<n}$ such that $\{ h_{ij}\}_{i,j<n}$ is diagonal, and $\forall \ i<n$
\begin{equation*}
    <e_i,E_{n+1}> =0.
\end{equation*}   
Since we assume $\eta^2 \log W $ is large, we have  
\begin{equation*}
    <\nu,E_{n+1}> =-\frac{1}{W}\rightarrow 0
\end{equation*}
and 
\begin{equation*}
    <e_n,E_{n+1}> = \sqrt{ 1 - <\nu,E_{n+1}>^2 } \rightarrow 1.
\end{equation*}
For the cutoff function, we have
\begin{equation*}
    \eta_i = - 2<X,e_i> +2<X,E_{n+1}> <e_i,E_{n+1}>.
\end{equation*}
and
\begin{equation*}
    \eta_{ij} = -2\delta_{ij}+ 2h_{ij}<X,\nu>  +2<e_j,E_{n+1}> <e_i,E_{n+1}>-2<X,E_{n+1}> h_{ij}<\nu,E_{n+1}>.
\end{equation*}
Then using 
\begin{equation*}
    X = <X,\nu> \nu + \sum_i <X,e_i>e_i,
\end{equation*}
we have 
\begin{equation*}
    \eta_n = -2<X,e_n><\nu,E_{n+1}>^2 + 2 <X,\nu> <\nu, E_{n+1}> <e_n, E_{n+1}> \rightarrow 0.
\end{equation*}
It is not hard to see that
\begin{equation*}
    \frac{G^{ij}\eta_{ij}}{\eta} -\frac{G^{ij}\eta_i \eta_j}{\eta^2} \geq - \frac{C}{\eta^2}.
\end{equation*}
Thus we have
\begin{eqnarray*}
    G^{ij}P_{ij}& \geq &-\frac{C}{\eta^2}-C +\frac{1}{2}\frac{ G^{ij}(\sum_k h_{ik}<e_k,E_{n+1}>)(\sum_l h_{jl}<e_l,E_{n+1}>)}{<\nu,E_{n+1}>^2 \log W}.
\end{eqnarray*}
The crucial term is estimated as follows:
\begin{eqnarray} 
    &\frac{ G^{ij}(\sum_k h_{ik}<e_k,E_{n+1}>)(\sum_l h_{jl}<e_l,E_{n+1}>)}{<\nu,E_{n+1}>^2 \log W}\nonumber\\
    =& \frac{G^{ij}h_{in}h_{jn}<e_n,E_{n+1}>^2}{<\nu,E_{n+1}>^2 \log W}\nonumber\\
    =& \frac{G^{nn}h^2_{nn}<e_n,E_{n+1}>^2}{<\nu,E_{n+1}>^2 \log W}+ 2 \sum_{i<n} \frac{G^{ni}h_{in}h_{nn}<e_n,E_{n+1}>^2}{<\nu,E_{n+1}>^2 \log W} \nonumber\\
    & +\sum_{i,j<n} \frac{G^{ij}h_{in}h_{jn}<e_n,E_{n+1}>^2}{<\nu,E_{n+1}>^2 \log W} \nonumber\\
    \geq & \frac{G^{nn}h^2_{nn}<e_n,E_{n+1}>^2}{<\nu,E_{n+1}>^2 \log W}+ 2\sum_{i<n} \frac{G^{ni}h_{in}h_{nn}<e_n,E_{n+1}>^2}{<\nu,E_{n+1}>^2 \log W}\label{eq: temp G^ij h_ik h_jl}.
\end{eqnarray}
Recall that $\frac{\eta_n}{\eta} \approx \frac{1}{W\eta}\rightarrow 0$, we have 
\begin{equation}\label{eq:temp h_nn}
    -\frac{h_{nn} <e_n,E_{n+1}>}{<\nu,E_{n+1}> \log W} = - \frac{2 \eta_n}{\eta} -  <e_n,E_{n+1}>\approx - 1, 
\end{equation}
and 
\begin{equation} \label{eq: temp h_ni}
    -\frac{h_{in} <e_n,E_{n+1}>}{<\nu,E_{n+1}> \log W} = - \frac{2 \eta_i}{\eta} = \frac{4<e_i,X>}{\eta}, \quad i<n.
\end{equation}
Thus we have 
\begin{equation}\label{eq:temp G_nn h_nn^2}
    \frac{G^{nn}h^2_{nn}<e_n,E_{n+1}>^2}{<\nu,E_{n+1}>^2 \log W} \geq \frac{1}{2} \log W G^{nn} .
\end{equation}
By the definition of $T_k$, for $i<n$, we have
\begin{eqnarray*}
    [T_k]^{ni} &=&  \frac{1}{k!} \delta^{i i_1 i_2 \cdots i_{k-1} i_k }_{nj_1j_2\cdots j_{k-1} j_k}  h_{i_1 j_1} h_{i_2 j_2} \cdots h_{i_{k-1} j_{k-1}} h_{i_k j_k}\\
    &=&  \frac{1}{(k-1)!} \delta^{i i_1 i_2 \cdots i_{k-1} n}_{nj_1j_2\cdots j_{k-1} i}  h_{i_1 j_1} h_{i_2 j_2} \cdots h_{i_{k-1} j_{k-1}} h_{n i}\\
    &=&  \frac{1}{(k-1)!} \delta^{in}_{ni} \sum_{ j_1,j_2,\cdots,j_{k-1} \neq i,n} \delta^{ j_1 j_2 \cdots j_{k-1} }_{j_1j_2\cdots j_{k-1} }  h_{j_1 j_1} h_{j_2 j_2} \cdots h_{j_{k-1} j_{k-1}} h_{n i}\\
     &=&  - h_{in}  \sum_{j_1,j_2,\cdots,j_{k-1} \neq i,n} \frac{1}{(k-1)!}   \delta^{ j_1 j_2 \cdots j_{k-1} }_{j_1j_2\cdots j_{k-1} }  h_{j_1 j_1} h_{j_2 j_2} \cdots h_{j_{k-1} j_{k-1}} \\
    & =& - h_{n i}\sigma_{k-1} (h \rvert n,i),
\end{eqnarray*}
where $h|n,i$ means the vector without $h_{ii}$ and $h_{nn}$.\\ 
And 

\begin{equation*}
    \lvert \sigma_{k-1} (h \rvert n,i)\rvert  \leq C(n) V.
\end{equation*}
Using \eqref{GF}, we have
\begin{equation}\label{eq: temp G^ni h_in}
    \lvert \sum_{i<n} G^{ni} h_{in} \rvert  \leq C(n) \sum_{i<n} \lvert h_{in} \rvert^2.   
\end{equation}
Because $P_i=0$, by (\ref{eq:temp h_nn}) we have
$$
    h_{nn} \approx  <\nu,E_{n+1}> \log W = -\frac{\log W}{W} \leq 0,
$$ and for $i<n$, by (\ref{eq: temp h_ni})
$$
    \lvert h_{in} \rvert \approx \frac{\log W}{\eta W } \le \frac{\log^2 W}{W}.
$$
By (\ref{eq: temp G^ni h_in})  we have 
$$
     \lvert \sum_{i<n} G^{ni} h_{in} \rvert  \leq C\frac{\log^4 W}{W^2}.
$$
Then we have 
\begin{equation}\label{eq: temp G^ni h_in h_nn}
    \sum_{i<n} \frac{G^{ni}h_{in}h_{nn}<e_n,E_{n+1}>^2}{<\nu,E_{n+1}>^2 \log W} \geq -C \frac{\log^4 W}{W}. 
\end{equation}
For some small constant $c=c(n)$, we have
\begin{equation}\label{eq: temp G^nn}
    G^{nn} \geq \frac{1}{1+h^2_{nn}+\sum_{i<n} h^2_{in}} \geq c.
\end{equation}
In all, by (\ref{eq: temp G^ij h_ik h_jl}) (\ref{eq:temp G_nn h_nn^2}) (\ref{eq: temp G^ni h_in h_nn}) (\ref{eq: temp G^nn}) we have 
\begin{eqnarray*}
     G^{ij}P_{ij} &\geq& \frac12\frac{G^{nn}h^2_{nn}<e_n,E_{n+1}>^2}{<\nu,E_{n+1}>^2 \log W}  +\frac12\sum_{i<n} \frac{G^{ni}h_{in}h_{nn}<e_n,E_{n+1}>^2}{<\nu,E_{n+1}>^2 \log W}\\
    & &- \frac{C}{\eta^2} - C\\
    & \geq  &\frac c2 \log W -C\frac{\log^4 W}{W}- \frac{C}{\eta^2} - C .
\end{eqnarray*}
Finally, we get
\begin{equation*}
    0\ge G^{ij}P_{ij} (x_0) \geq \frac{c}{4} \log W (x_0)-\frac{C}{\eta^2(x_0)} -C.
\end{equation*}
Thus for $\forall x \in B_1$, we have 
$$
\eta^2 \log W(x) \leq \eta^2\log W(x_0)\leq C.
$$
\end{proof}




\end{document}